\documentclass{article}%
\usepackage{fullpage}
\usepackage{amsmath}
\usepackage{amsfonts}
\usepackage{amssymb}
\usepackage{amsthm}
\usepackage{graphicx}%
\providecommand{\U}[1]{\protect\rule{.1in}{.1in}}
\newtheorem{definition}{Definition}
\newtheorem{theorem}{Theorem}

\newtheorem{lemma}[theorem]{Lemma}
\newtheorem{conjecture}[theorem]{Conjecture}
\newtheorem{notation}{Notation}
\newtheorem{example}{Example}
\newtheorem{remark}{Remark}
\newtheorem{algorithm}{Algorithm}

\allowdisplaybreaks

\usepackage{color}

\begin{document}

\title{Lower Bounds for Maximum Gap in\\(Inverse) Cyclotomic Polynomials}
\author{Mary Ambrosino, Hoon Hong, Eunjeong Lee}
\maketitle

\begin{abstract}
The maximum gap $g(f)$ of a polynomial $f$ is the maximum of the differences
(gaps) between two consecutive exponents that appear in $f$. Let $\Phi_{n}$
and $\Psi_{n}$ denote the $n$-th cyclotomic and $n$-th inverse cyclotomic
polynomial, respectively. In this paper, we give several lower bounds for
$g(\Phi_{n})$ and $g(\Psi_{n})$, where $n$ is the product of odd primes. We
observe that they are very often exact. We also give an exact expression for
$g(\Psi_{n})$ under a certain condition. Finally we conjecture an exact
expression for $g(\Phi_{n})$ under a certain condition.

\end{abstract}

\section{Introduction}

The $n$-th cyclotomic and $n$-th inverse cyclotomic polynomials are defined as follows%

\[
\Phi_{n}(x)=\prod_{\substack{1\leq k\leq n\\(k,n)=1}}\left(  x-e^{2\pi
i\frac{k}{n}}\right)  \hspace{40pt}\Psi_{n}(x)=\prod_{\substack{1\leq k\leq
n\\(k,n)\neq1}}\left(  x-e^{2\pi i\frac{k}{n}}\right)
\]
For example, we have%
\begin{align*}
\Phi_{15}(x)  &  =1-x+x^{3}-x^{4}+x^{5}-x^{7}+x^{8}\\
\Psi_{15}(x)  &  =-1-x-x^{2}+x^{5}+x^{6}+x^{7}%
\end{align*}

\noindent There have been extensive studies on the coefficients of cyclotomic
polynomials
\cite{Lehmer1936,Beiter1978,Bachman2003,Dresden2004,Kaplan2007,Gallot2009,
Gallot2009_2,Thangadurai2000,Bzdega2010,Gallot2011,Fintzen2011,
Bzdega2012,Moree2012_2,Cobeli2013,Fouvry2013, Bzdega2014,Bzdega2014_2}, and
more recently, on inverse cyclotomic polynomials
\cite{Moree2009,Bzdega2010,Bzdega2014_2}.

In \cite{Hong2012}, a study was initiated on their exponents, in particular on
the maximum gap $g$, that is, the largest difference between consecutive
exponents: for example, $g(\Phi_{15})=2$ since $2$ is the maximum among $1-0$,
$3-1$, $4-3$, $5-4$, $7-5$, $8-7$. The original motivation came from elliptic
curve cryptography; the computing time of the $\text{Ate}_{i}$ pairing over
elliptic curves depends on the maximum gap of the inverse cyclotomic
polynomials whose degree are decided from the parameter of the elliptic
curves~\cite{Zhao2008,Lee2009,Thangadurai2000,Hong2013}. However the problem
of finding the maximum gap is interesting on its own and its study can be
viewed as a first step toward the detailed understanding of the sparsity
structure of $\Phi_{n}$ and $\Psi_{n}$.

One can restrict the problem to the case when $n$ is a product of odd primes,
because all other cases can be trivially reduced to it (section 2
of~\cite{Hong2012}). Thus, let us assume that $n=p_{1}\cdots p_{k}$ where
$p_{1}<\cdots<p_{k}$ are odd primes. It is obvious that $g(\Phi_{p_{1}%
})=g(\Psi_{p_{1}})=1$. It is also obvious that $g(\Psi_{p_{1}p_{2}}%
)=p_{2}-(p_{1}-1)$. Hence, the simplest non-trivial cases are $g(\Phi
_{p_{1}p_{2}})$ and $g(\Psi_{p_{1}p_{2}p_{3}})$. In~\cite{Hong2012}, it was
shown that $g(\Phi_{p_{1}p_{2}})=p_{1}-1$ and that $g(\Psi_{p_{1}p_{2}p_{3}%
})=2p_{2}p_{3}-\psi(p_{1}p_{2}p_{3})$ under a certain mild condition, where
$\psi(n)=\deg(\Psi_{n})$. Since then, several simpler or more insightful
proofs were found along with other interesting properties
\cite{Moree2012,Zhang2015,Camburu2015}.

Naturally, the next challenge is to find general expressions for $g(\Phi_{n})$
and $g(\Psi_{n})$ where $n$ is the product of an \emph{arbitrary} number of
odd primes. However, after several years of attempts, we have not yet found
any general expressions, due to combinatorial blowup in the number of cases to
consider. Thus, we propose to consider instead a weaker challenge: find
expressions for \emph{lower bounds} of $g(\Phi_{n})$ and $g(\Psi_{n})$. The
weaker challenge is still useful for the original motivation from elliptic
curve cryptography.

Thus, in this paper, we tackle the weaker challenge of finding expressions for
lower bounds. The main contributions (precisely stated in
Section~\ref{sec:results}) are as follows.

\begin{enumerate}
\item We provide four expressions ($\alpha^{\pm},\beta^{\pm}$, $\gamma^{\pm}$
and $\delta^{-}$) for lower bounds (Theorems~\ref{thm:alpha},~\ref{thm:beta}, \ref{thm:gamma}
and \ref{thm:delta}). These expressions were discovered by carefully
inspecting and finding patterns among the maximum gaps of many cyclotomic and
inverse cyclotomic polynomials. The four expressions are easy to compute.
Furthermore, numerous computer experiments indicate that the combination
(maximum) of the four expressions is \emph{very often} exact
(Section~\ref{sec:quality_alpha_beta_gamma_delta}).

\item We abstract the four expressions into a single general expression
$\varepsilon^{\pm}$ (Theorem~\ref{thm:epsilon}). The general expression was
discovered by observing that each of the four expressions can be rewritten as
the difference of two numbers, say $u$ and $l$, where $u$ is a certain divisor
of $n$ and $l$ is a signed sum of several other divisors of $n$. We also
observed that there is indeed a gap between $x^{l}$ and $x^{u}$ in the
polynomials, which led to an idea for proving the general expression. The
general expression takes more time to compute, since it captures many other
gaps that are not captured by the four expressions. As a result,
$\varepsilon^{\pm}$ is always greater than or equal to $\alpha^{\pm}%
,\ \beta^{\pm}$, $\gamma^{\pm}$ and $\delta^{-}$. Indeed, numerous computer
experiments indicate that it is \emph{almost always} exact
(Section~\ref{sec:quality_e}).

\item We provide a sufficient condition that $g(\Psi_{n})=\delta^{-}%
$~(Theorem~\ref{thm:condition_delta}). It is a straightforward generalization
of a result in \cite{Hong2012} for the case $k=3$. We also show that, for
every fixed $p_{1}$, the sufficient condition holds \textquotedblleft almost
always\textquotedblright\ in a certain sense.

\item Finally we conjecture that $g(\Phi_{p_{1}\cdots p_{k}})=\varphi
(p_{1}\cdots p_{k-1})$ if and only if $p_{k}>p_{1}\cdots p_{k-1}$
(Conjecture~\ref{conj:condition_phi}). It is a natural generalization of the
result in \cite{Hong2012}: $g(\Phi_{p_{1}p_{2}})=p_{1}-1=\varphi(p_{1})$. The
conjecture has been already verified for $m=p_{1}\cdots p_{k-1}<1000$ and
arbitrary $p_{k}$ (Theorem~\ref{thm:verified}). The verification technique is
based on a structural result that $g(\Phi_{mp_{k}})$ only depends on $m$ and
$\mathrm{rem}(p_{k},m)$~(Theorem~\ref{thm:modulo}). Thus, given $m$, we only
need to check finitely many $p_{k}$ values in order to check the conjecture
for infinitely many~$p_{k}$. We organized it into an algorithm
(Algorithm~\ref{alg:modulo}) and ran it for all odd square-free $m<1000$.
\end{enumerate}

\noindent The paper is structured as follows: In Section~\ref{sec:results}, we
precisely state the lower bounds and the conjecture informally described
above. In Section~\ref{sec:example}, we illustrate each bound using small
examples. In Section~\ref{sec:quality}, we report experimental findings on the
quality of the bounds (how often they are exact). In Section~\ref{sec:proof},
we prove the lower bounds. In Section~\ref{sec:support}, we provide supporting
evidence for the conjecture.

\section{Main Results}

\label{sec:results}

In this section, we precisely state the main results of this paper. From now
on, let $n=p_{1}\cdots p_{k}$ where $p_{1}<\cdots<p_{k}$ are odd primes.
Recall several standard notations. For a square-free $d$, $\varphi
(d)=\deg(\Phi_{d})$, $\psi(d)=\deg(\Psi_{d})$, $\omega(d)=$ number of prime
factors of $d$, and $\mu(d)=(-1)^{\omega(d)}$. For an integer $i$,
$\rho\left(  i\right)  $ is the parity, that is $\left(  -1\right)  ^{i}$. We
formally define the maximum gap as follows:

\begin{definition}
[Maximum gap]Let $f(x)=c_{1}x^{\nu_{1}}+\cdots+c_{t}x^{\nu_{t}}$ where
$c_{1},\ldots,c_{t}\neq0$ and $0\leq\nu_{1}<\cdots<\nu_{t}$. Then the maximum
gap of $f$, denoted $g(f)$, is defined as follows
\[
g(f)=\max_{1\leq i<t}(\nu_{i+1}-\nu_{i})
\]
if $t\neq1$, and $g(f)=0$ if $t=1$.
\end{definition}

\noindent Now we are ready to state the four lower bounds for (inverse)
cyclotomic polynomials.

\begin{theorem}
[Special bound $\alpha^{\pm}$]\label{thm:alpha} We have $g(\Phi_{n})\geq
\alpha^{+}(n)$ and $g(\Psi_{n})\geq\alpha^{-}(n)$ where%
\[
\alpha^{\pm}(n)=\max_{\substack{1\leq r<k\\\rho\left(  k-r\right)  =\mp
1}}\left(  p_{r}-\varphi(p_{1}\cdots p_{r-1})\right)
\]

\end{theorem}

\begin{theorem}
[Special bound $\beta^{\pm}$]\label{thm:beta} We have $g(\Phi_{n})\geq
\beta^{+}(n)$ and $g(\Psi_{n})\geq\beta^{-}(n)$ where%
\[
\beta^{\pm}(n)=\max_{\substack{1\leq r<k\\\rho\left(  k-r\right)  =\mp
1}}\left(  \min\left\{  p_{r+1},p_{1}\cdots p_{r}\right\}  -\psi(p_{1}\cdots
p_{r})\right)
\]

\end{theorem}

\begin{theorem}
[Special bound $\gamma^{\pm}$]\label{thm:gamma} We have $g(\Phi_{n})\geq
\gamma^{+}(n)$ and $g(\Psi_{n})\geq\gamma^{-}(n)$ where%
\[
\gamma^{\pm}(n)=\max_{\substack{1\leq r<k\\\rho\left(  k-r\right)  =\mp
1}}\left(  p_{1}\cdots p_{r}-\sum_{\substack{d|n\\\omega\left(  d\right)
<r}}\pm\mu\left(  n/d\right)  \ d\right)
\]

\end{theorem}

\begin{theorem}
[Special bound $\delta^{-}$]\label{thm:delta} We have $g(\Psi_{n})\geq
\delta^{-}(n)$ where%
\[
\delta^{-}(n)=2\frac{n}{p_{1}}-\psi(n)
\]

\end{theorem}

\noindent Now we describe a more general lower bound, which is abstracted from
the above four bounds. For this, we need a few notations.

\begin{notation}
For a positive integer $d$ and a set $B$ of positive integers, let
\[%
\begin{array}
[c]{lllllll}%
\overline{d} & = & \left\{  h\;:\;d\ |\ h\right\}  &  & \underline{B} & = &
\displaystyle\bigcup_{d\in B}\underline{d}\\
\underline{d} & = & \left\{  h\;:\;h\ |\ d\right\}  &  & B^{\pm} & = &
\left\{  d\in B\;:\;\mu\left(  n/d\right)  =\pm1\right\}
\end{array}
\]

\end{notation}

\noindent Now are ready to state the general bound, unifying the four special bounds.

\begin{theorem}
[General bound $\varepsilon^{\pm}$]\label{thm:epsilon} We have $g(\Phi
_{n})\geq\varepsilon^{+}(n)$ and $g(\Psi_{n})\geq\varepsilon^{-}(n)$ where%
\[
\varepsilon^{\pm}(n)=\max_{\substack{A\uplus B=\underline{n}\setminus
\{n\}\\A\neq\emptyset\\\mathcal{C}^{\pm}(B)}}\left(  \min A\ -\ 
\sum_{d\in B}\pm \mu\left(n/d\right)\  d \right)
\]
where
\[
\mathcal{C}^{\pm}(B)\;\;\;\Leftrightarrow\;\;\;\forall d\in\underline{B}%
\;\ \ \;\;\#(B^{\pm}\cap\overline{d})\;\geq\;\#(B^{\mp}\cap\overline{d})
\]

\end{theorem}

\begin{remark}
The above four special bounds $\alpha^{\pm},\ \beta^{\pm}$, $\gamma^{\pm}$ and
$\delta^{-}$ can be obtained from the general bound $\varepsilon^{\pm}$ by
considering only certain $B$'s:

\begin{enumerate}
\item[$\alpha^{\pm}$:] $B=\left\{  d:d|p_{1}\cdots p_{r-1}\ \text{and }%
\omega\left(  d\right)  <r\right\}  $\;\;\;\;\;\; for $1\leq r<k$ and $\rho\left(
k-r\right)  =\mp1$

\item[$\beta^{\pm}$:] $B=\left\{  d:d|p_{1}\cdots p_{r}\;\;\;\;\; \text{and }%
\omega\left(  d\right)  <r\right\}  $\;\;\;\;\;\; for $1\leq r<k$ and $\rho\left(
k-r\right)  =\mp1$

\item[$\gamma^{\pm}$:] $B=\left\{  d:d|p_{1}\cdots p_{k}\;\;\;\;\; \text{and }%
\omega\left(  d\right)  <r\right\}  $\;\;\;\;\;\; for $1\leq r<k$ and $\rho\left(
k-r\right)  =\mp1$

\item[$\delta^{-}$:] $B=\left\{  d:d|p_{1}\cdots p_{k}\;\;\;\;\; \text{and }%
\omega\left(  d\right)  <k\;\;\; \text{and }d\neq p_{2}\cdots p_{k}\right\}  $
\end{enumerate}

\noindent It turns out that these $B$'s satisfy $\mathcal{C}^{\pm}(B)$.
\end{remark}

\begin{theorem}
[Sufficient condition on $g(\Psi_{n})$]\label{thm:condition_delta} We have

\begin{enumerate}
\item $g(\Psi_{n})=\delta^{-}(n)$ if $\delta^{-}(n)\geq\frac{1}{2}\frac
{n}{p_{1}}$.

\item For every $k\geq2$ and every odd prime $p$, we have
\[
\lim_{b\rightarrow\infty}\frac{\#\left\{  n\;:\;p_{k}\leq b,\,p_{1}%
=p,\,\delta^{-}(n)\geq\frac{1}{2}\frac{n}{p_{1}}\right\}  }{\#\left\{
\,n\;:\;p_{k}\leq b,\,p_{1}=p\,\right\}  \hfill}=1
\]

\end{enumerate}
\end{theorem}

\begin{conjecture}
[Equivalent condition on $g(\Phi_{n})$]\label{conj:condition_phi} We have
\[
g(\Phi_{n})=\varphi(p_{1}\cdots p_{k-1})\ \ \text{if and only if \ }%
p_{k}>p_{1}\cdots p_{k-1}%
\]

\end{conjecture}

\section{Examples}

\label{sec:example}

\subsection{Examples for the bound $\alpha^{\pm},\beta^{\pm},\gamma^{\pm
},~\delta^{-}$ and $\varepsilon^{\pm}$ (Theorems~\ref{thm:alpha},
\ref{thm:beta}, \ref{thm:gamma}, \ref{thm:delta} and \ref{thm:epsilon})}

In the following two tables, we give the values of $g(\Phi_{n})$, $g(\Psi_{n})$
 and the lower bounds on several
values of $n$.%
\begin{align*}
&
\begin{array}
[c]{c||c|c|c|c|c}\hline
n & 3\cdot5\cdot11\cdot 13
 & 3\cdot5\cdot7\cdot 71 & 7\cdot 11\cdot 13\cdot 17 & 3\cdot7\cdot11\cdot13 &
3\cdot5\cdot7\cdot11\\\hline\hline
g(  \Phi_{n})  & \mathbf{3} & \mathbf{14} & \mathbf{210} &
\mathbf{17} & \mathbf{10}\\
\alpha^{+}(  n)  & \mathbf{3} & 2 & 6 & 2 & 2\\
\beta^{+}(  n)  & 2 & \mathbf{14} & 6 & 2 & 2\\
\gamma^{+}(  n)  & 2 & 2 & \mathbf{210} & 2 & 2\\
\varepsilon^{+}(  n)  & \mathbf{3} & \mathbf{14} & \mathbf{210} &
\mathbf{17} & 2\\\hline
\end{array}
\\
\\
&
\begin{array}
[c]{c||c|c|c|c|c}\hline
n & 5\cdot7\cdot11\cdot13 & 7\cdot 11\cdot 13\cdot 17
 & 3\cdot5\cdot7\cdot 11 & 
 7\cdot 11\cdot 13\cdot 41 &7\cdot11\cdot13\\\hline\hline
g(  \Psi_{n})  & \mathbf{3} & \mathbf{30} & \mathbf{95}
 & \mathbf{11} & \mathbf{7}\\
\alpha^{-}(  n)  & \mathbf{3} & 5 & 3 & 5 & 6\\
\beta^{-}(  n)  & 0 & -4 & 0 & 4 & 6\\
\gamma^{-}(  n)  & 0 & \mathbf{30} & -10 & 6 & 6\\
\delta^{-}(  n)  & -123 & -635 & \mathbf{95} & -515 & 5\\
\varepsilon^{-}(  n)  & \mathbf{3} & \mathbf{30} & \mathbf{95} & \mathbf{11} & 6\\\hline
\end{array}
\end{align*}
In the above tables, we marked the exact ones in boldface, that is, the ones
that match $g(  \Phi_{n})$ or $g(  \Psi_{n})$.
For the last column, we chose the smallest $n$ such that
 $g(\Phi_{n})$ and $g(\Psi_{n})$ is not equal to any of the lower bounds.
After checking all the values of $n < 15013$, we have not found any such example
for the cyclotomic case
where $k=3$.

 In the following, we will illustrate how the above bounds
  are computed for some of
the examples.

\begin{example}
[$\alpha^{+}$]Let $n=3\cdot5\cdot11\cdot 13$.
 We will compute $\alpha^{+}(n)$. Let%
\begin{align*}
u &  =p_{r}\\
l &  =\varphi(p_{1}\cdots p_{r-1})
\end{align*}
The following table shows the values of $u-l$ for all choices of $r$ such that
$1\leq r<k\ $and $\delta(  k-r)  =-1$.%
\[%
\begin{array}
[c]{l|l|l|l}%
r & u & l & u-l\\  \hline
1& 3 & 1 & 2 \\
3 & 11 & 8 & 3
\end{array}
\]
Thus $\alpha^{+}(n)= 3$. 
\end{example}

\begin{example}
[$\beta^{+}$]Let $n=3\cdot5\cdot 7\cdot71$. We will compute $\beta^{+}(n)$. Let%
\begin{align*}
u &  =\min\left\{  p_{r+1},p_{1}\cdots p_{r}\right\}  \\
l &  =\psi(p_{1}\cdots p_{r})
\end{align*}
The following table shows the values of $u-l$ for all choices of $r$ such that
$1\leq r<k\ $and $\delta\left(  k-r\right)  =-1$.%
\[%
\begin{array}
[c]{l|l|l|l}%
r & u & l & u-l\\  \hline
1 & 3 & 1 & 2\\
3 & 71 & 57 & 14
\end{array}
\]
Thus $\beta^{+}(n)= 14$. 
\end{example}

\begin{example}
[$\gamma^{+}$]Let $n=7\cdot11\cdot 13\cdot 17$.
 We will compute $\gamma^{+}(n)$. Let
\begin{align*}
u &  =p_{1}\cdots p_{r}\\
B &  =\left\{  d\;:\;d\ |\ n\text{ and }\omega\left(  d\right)  <r\right\}  \\
l &  =\sum_{d\in B}\mu\left(  n/d\right)  \ d
\end{align*}
The following table shows the values of $u-l$ for all choices of $r$ such that
$1\leq r<k$ and $\delta\left(  k-r\right)  =-1$.%
\[%
\begin{array}
[c]{l|l|l|l|l}%
r & u & B & l & u-l\\ \hline
1& 7 & \{1 \}  &  1 & 6 \\
3 & 7\cdot 11\cdot 13 & \{1,\, 7,\, 11,\, 13,\, 17,\,
 77,\, 91,\, 119,\, 143,\, 187,\, 221\} & 791 & 210
\end{array}
\]
Thus $\gamma^{+}(n)=210$. 
\end{example}

\begin{example}
[$\varepsilon^{+}$]Let $n=3\cdot7\cdot11\cdot13$. We will compute
$\varepsilon^{+}(n)$. Let
\begin{align*}
u  &  =\min A\\
l  &  =\sum_{d\in B}\mu\left(n/d\right) \ d
\end{align*}
The following table shows the values of $u-l$ for some $A$ and $B$
such that $A\uplus B=\underline{n}\setminus\{n\}$, $A\neq\emptyset$, and
$\mathcal{C}^{+}(B)$.
There are 1566 such pairs of  $A$ and $B$, so we only list a few below.
\[%
\begin{array}
[c]{l|l|l|l|l}%
A & B & u & l & u-l\\\hline
\{3,\, 7,\, 11,\, 13,\, 3\cdot 7,\,\ldots \}
&  \{1\}  &  3    &  1  &   2    \\
\{11,\, 13,\, 3\cdot 11,\, 3\cdot 13,\, 7\cdot 11
,\, \ldots \}
 & \{1,\, 3,\, 7,\, 3\cdot 7\}   &   11 &  12  &  -1   \\
\{13,\, 3\cdot 13,\, 7\cdot 11,\, 7\cdot 13,\,\ldots \} 
& \{1,\, 3,\, 7,\, 11,\, 3\cdot 7,\, 3\cdot 11\}
  &  13    &   34   &    -21   \\
\{7\cdot 11,\, 7\cdot 13,\, 11\cdot 13,\,\ldots \}&
  \{1,\, 3,\, 7,\, 11,\, 13,\, 3\cdot 7,\, 3\cdot 11,\, 3\cdot 13\}
 &  7\cdot 11   &   60  &    17  \\
\cdots & \cdots 
\end{array}
\]
Thus $\varepsilon^{+}(n)=17$.
\end{example}

\begin{example}
[$\alpha^{-}$]Let $n=5\cdot7\cdot11\cdot13$. 
We will compute $\alpha^{-}(n)$. Let%
\begin{align*}
u &  =p_{r}\\
l &  =\varphi(p_{1}\cdots p_{r-1})
\end{align*}
The following table shows the values of $u-l$ for all choices of $r$ such that
$1\leq r<k\ $and $\delta\left(  k-r\right)  =+1$.%
\[%
\begin{array}
[c]{l|l|l|l}%
r & u & l & u-l\\ \hline
2& 7 & 4 &3
\end{array}
\]
Thus $\alpha^{-}(n)= 3$. 
\end{example}

\begin{example}
[$\gamma^{-}$]Let $n=7\cdot11\cdot13\cdot 17$.
We will compute $\gamma^{-}(n)$. Let%
\begin{align*}
u &  =p_{1}\cdots p_{r}\\
B &  =\left\{  d\;:\;d\ |\ n\text{ and }\omega\left(  d\right)  <r\right\}
 \\
l &  =\sum_{d\in B}-\mu\left(  n/d\right)  \ d
\end{align*}
The following table shows the values of $u-l$ for all choices of $r$ such
that
$1\leq r<k$ and $\delta\left(  k-r\right)  =+1$.%
\[%
\begin{array}
[c]{l|l|l|l|l}%
r & u & B & l & u-l\\ \hline
2& 7\cdot 11 & \{1,\, 7,\, 11,\, 13,\, 17 \}  &  47 & 30
\end{array}
\]
Thus $\gamma^{-}(n)= 30$. 
\end{example}

\begin{example}
[$\delta^{-}$]
\label{ex:middle}Let $n=3\cdot5\cdot7$. We will compute $\delta^{-}(n)$. Note%
\begin{align*}
\delta^{-}(n)  & =2\frac{n}{p_{1}}-\psi(n)\\
& =2\frac{3\cdot5\cdot7}{3}-\left(  3\cdot5\cdot7-\left(  3-1\right)  \left(
5-1\right)  \left(  7-1\right)  \right)  \\
& =13
\end{align*}
Thus $\delta^{-}\left(  n\right)  =13$.
\end{example}

\begin{example}
[$\varepsilon^{-}$]Let $n=7\cdot11\cdot13\cdot41$. We will compute
$\varepsilon^{-}(n)$. Let
\begin{align*}
u  &  =\min A\\
l  &  =\sum_{d\in B}-\mu\left(n/d\right) \ d
\end{align*}
The following table shows the values of $u-l$ for some $A$ and
$B$
such that $A\uplus B=\underline{n}\setminus\{n\}$, $A\neq\emptyset$, and
$\mathcal{C}^{-}(B)$.
There are 13301 such pairs of $A$ and $B$, so we only list a few below.
\[%
\begin{array}
[c]{l|l|l|l|l}%
A & B & u & l & u-l\\\hline
\{11,\, 13,\, 41,\, 7\cdot 11,\, 7\cdot 13,\,\ldots\}
&  \{1,\, 7\}        &   11 &    6   &    5    \\
 \{1,\, 41, \, 7\cdot 11,\, 7\cdot 13,\, 11\cdot 13,\,\ldots \}&
   \{7,\, 11,\, 13\}    &      1    &   31 &     -30  \\
\{41,\, 7\cdot 11,\, 7\cdot 13,\,  11\cdot 13,\, \ldots \} &
 \{1,\, 7,\, 11,\, 13\}   &    41    &  30    &  11   \\
 \cdots & \cdots 
\end{array}
\]
Thus $\varepsilon^{-}(n)=11$.
\end{example}

\subsection{Examples for Sufficient condition on $g(\Psi_{n})$
(Theorem~\ref{thm:condition_delta})}

\begin{example}
Let $n=3\cdot7\cdot11$. Then $\delta^{-}(n)=43$. Consider
\[
\frac{1}{2}\left(  \frac{3\cdot7\cdot11}{3}\right)  =\frac{77}{2}\leq
\delta^{-}(n)
\]
Computation of $\Psi_{n}$ shows that $g(\Psi_{n})=43$, as expected from the theorem.
\end{example}

\begin{example}
Let $n=3\cdot5\cdot7$. In Example~\ref{ex:middle} we showed that $\delta
^{-}(n)=13$ and $g(\Psi_{n})=13$. Consider the following
\[
\frac{1}{2}\left(  \frac{3\cdot5\cdot7}{3}\right)  =\frac{35}{2}>\delta
^{-}(n)
\]
Therefore, the condition is sufficient but not necessary.
\end{example}

\begin{example}
Let $n=7\cdot11\cdot13$. Then $\delta^{-}(n)=5$. Consider
\[
\frac{1}{2}\left(  \frac{7\cdot11\cdot13}{7}\right)  =\frac{143}{2}>\delta
^{-}(n)
\]
Computation of $\Psi_{n}$ shows that $g(\Psi_{n})=6$. Thus $\delta^{-}(n)\neq
g(\Psi_{n})$.
\end{example}



\section{Quality}

\label{sec:quality}

\subsection{Quality of Special bounds $\alpha^{\pm}$, $\beta^{\pm}$,
$\gamma^{\pm}$ and $\delta^{-}$ (Theorems~\ref{thm:alpha},\ \ref{thm:beta}%
,~\ref{thm:gamma} and~\ref{thm:delta})}

\label{sec:quality_alpha_beta_gamma_delta}

The following graphs show how often the lower bound is equal to the maximum gap.%

\[\hspace{-40pt}
\begin{array}
[c]{ccc}%
f^{+}%
(b)\raisebox{-0.5\height}{\includegraphics[height=1.2in]{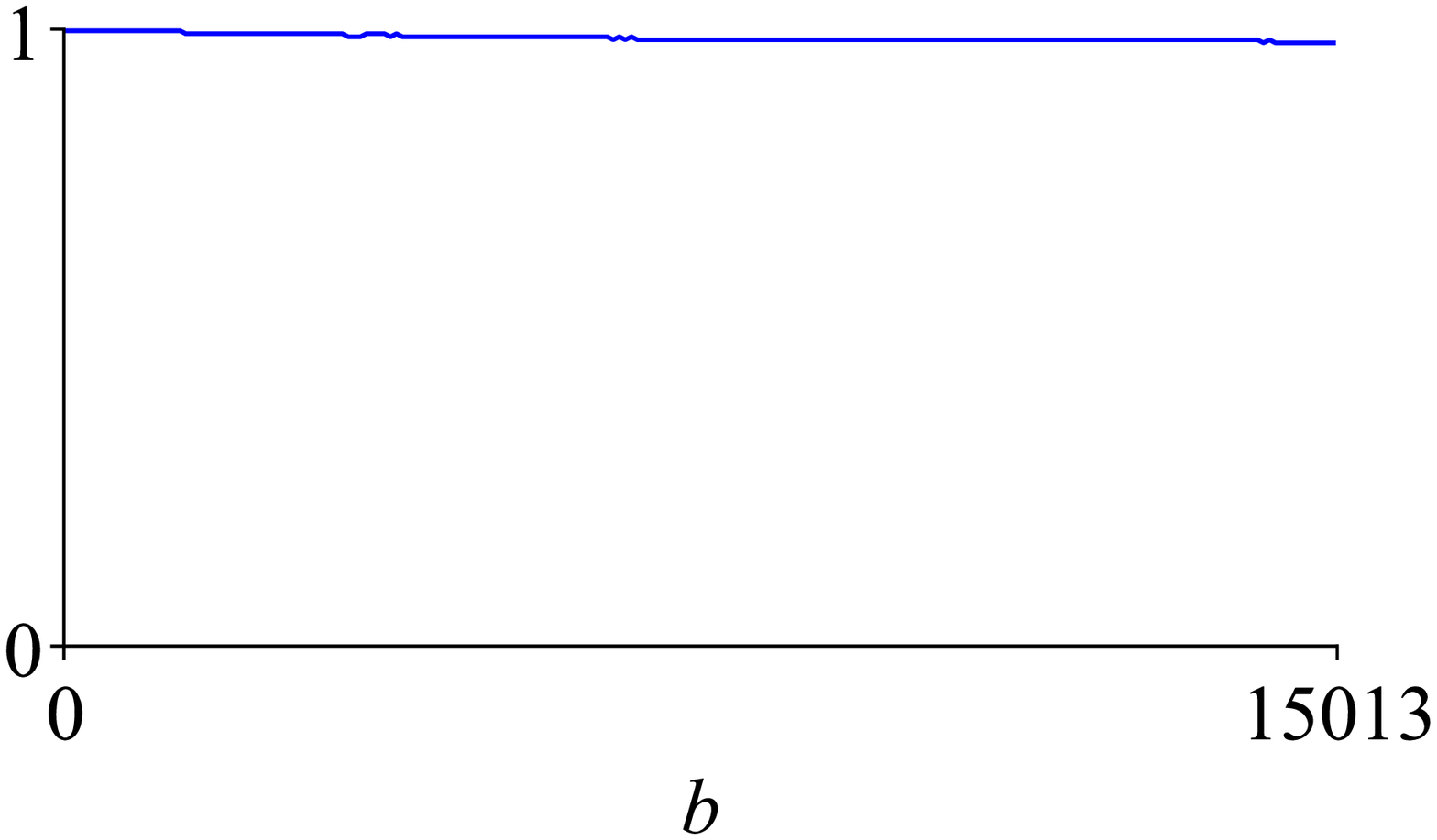}} &
& f^{-}%
(b)\raisebox{-0.5\height}{\includegraphics[height=1.2in]{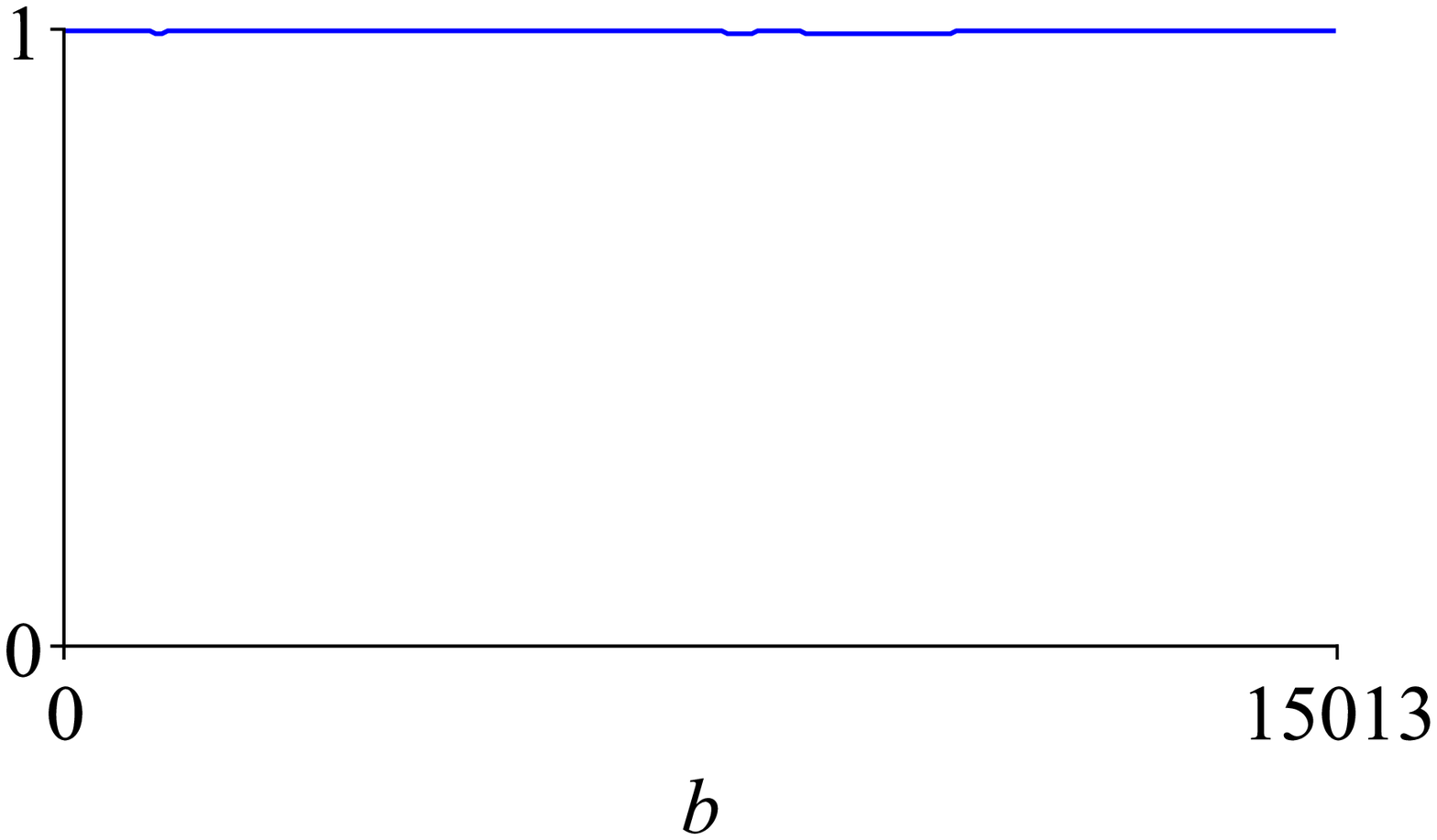}}\\
\displaystyle f^{+}(b)=\frac{\#\left\{  n<b\;:\;g(\Phi_{n})=\max\{\alpha
^{+}(n),\,\beta^{+}(n),\,\gamma^{+}(n)\}\right\}  }{\#\left\{  n<b\right\}
\hfill} &  & \displaystyle f^{-}(b)=\frac{\#\left\{  n<b\;:\;g(\Psi_{n}%
)=\max\{\alpha^{-}(n),\,\beta^{-}(n),\,\gamma^{-}(n),\,\delta^{-}%
(n)\}\right\}  }{\#\left\{  n<b\right\}  \hfill}%
\end{array}
\]
\noindent In the above graphs, $f^{+}(15013)=0.9829$
 and $f^{-}(15013)= 0.9984$.

\subsection{Quality of General bound $\varepsilon^{\pm}$
(Theorem~\ref{thm:epsilon})}

\label{sec:quality_e}

The following graphs show how often the lower bound is equal to the maximum gap.%

\[%
\begin{array}
[c]{ccc}%
f^{+}%
(b)\raisebox{-0.5\height}{\includegraphics[height=1.2in]{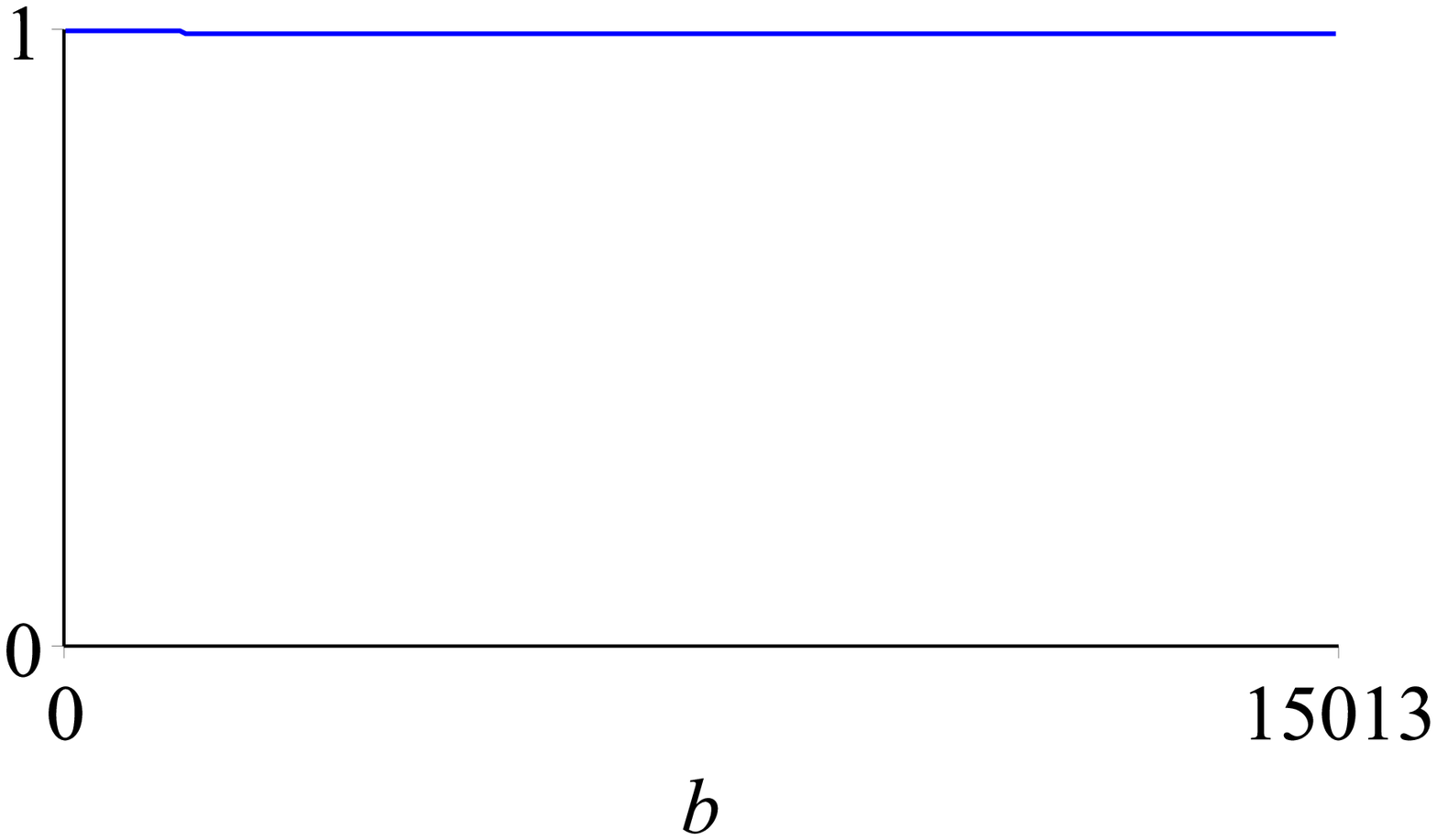}} &
& f^{-}%
(b)\raisebox{-0.5\height}{\includegraphics[height=1.2in]{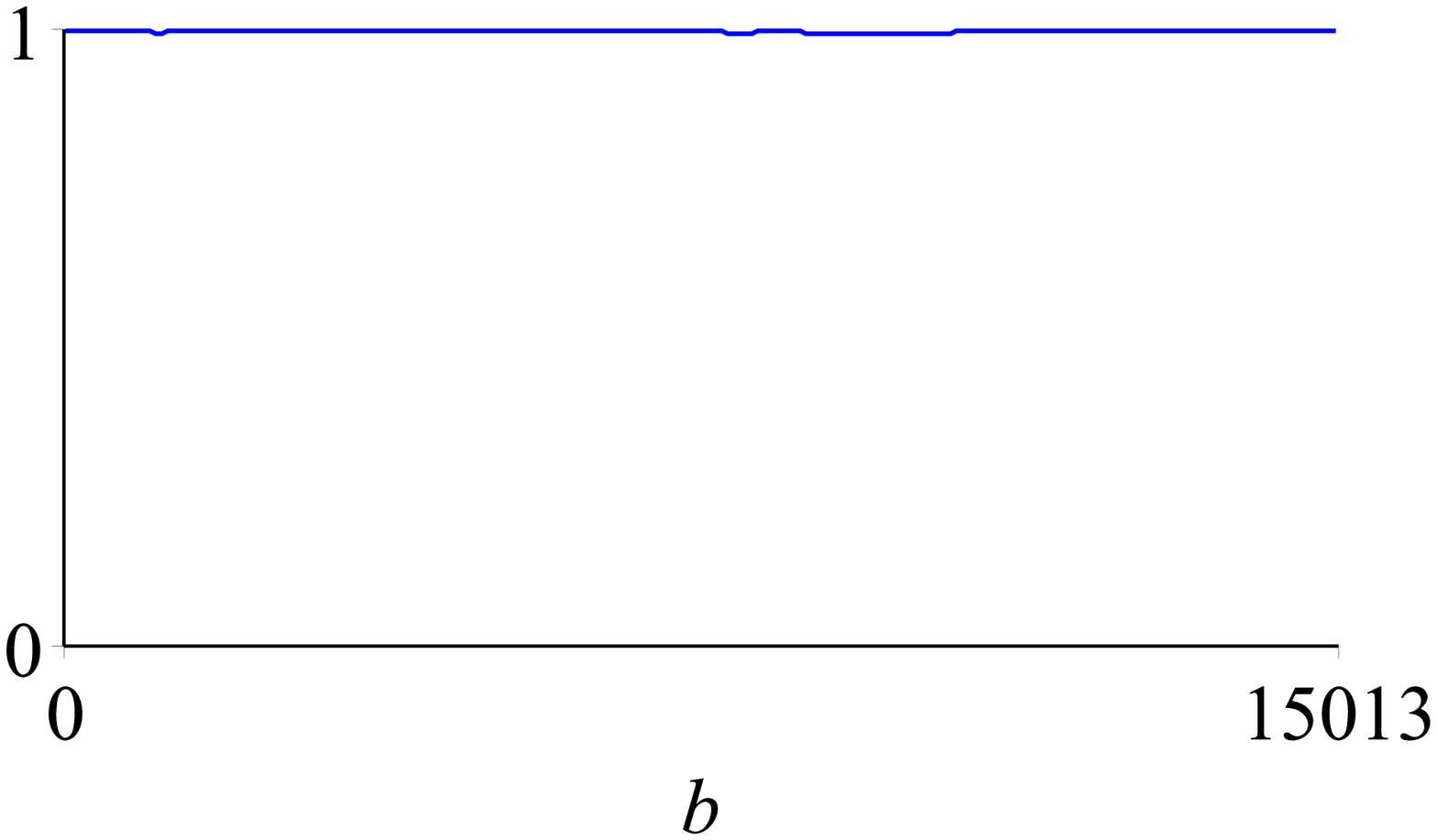}}\\
\displaystyle f^{+}(b)=\frac{\#\{n<b\;:\;g(\Phi_{n})=\varepsilon^{+}%
(n)\}}{\#\{n<b\}\hfill} &  & \displaystyle f^{-}(b)=\frac{\#\{n<b\;:\;g(\Psi
_{n})=\varepsilon^{-}(n)\}}{\#\{n<b\}\hfill}%
\end{array}
\]

\noindent In the above graphs, $f^{+}(15013)=0.9957$ and $f^{-}(15013)=0.9984
$.

\subsection{Quality of Sufficient condition on $g(\Psi_{n})$
(Theorem~\ref{thm:condition_delta})}

The following plots show the following ratio for various values of $k$ and $p
$.%

\[
r=\frac{\#\left\{  n\;:\;p_{k}\leq b,\,p_{1}=p,\,\delta^{-}(n)\geq\frac{1}%
{2}\frac{n}{p_{1}}\right\}  }{\#\left\{  \,n\;:\;p_{k}\leq b,\,p_{1}%
=p\right\}  \hfill}%
\]

\noindent We observe that in all cases, the ratio goes to 1, as expected from
the theorem.%

\[%
\begin{array}
[c]{ccc}%
k=2,\, p=3 &  & k=2,\, p=11\\
\includegraphics[height=1.2in]{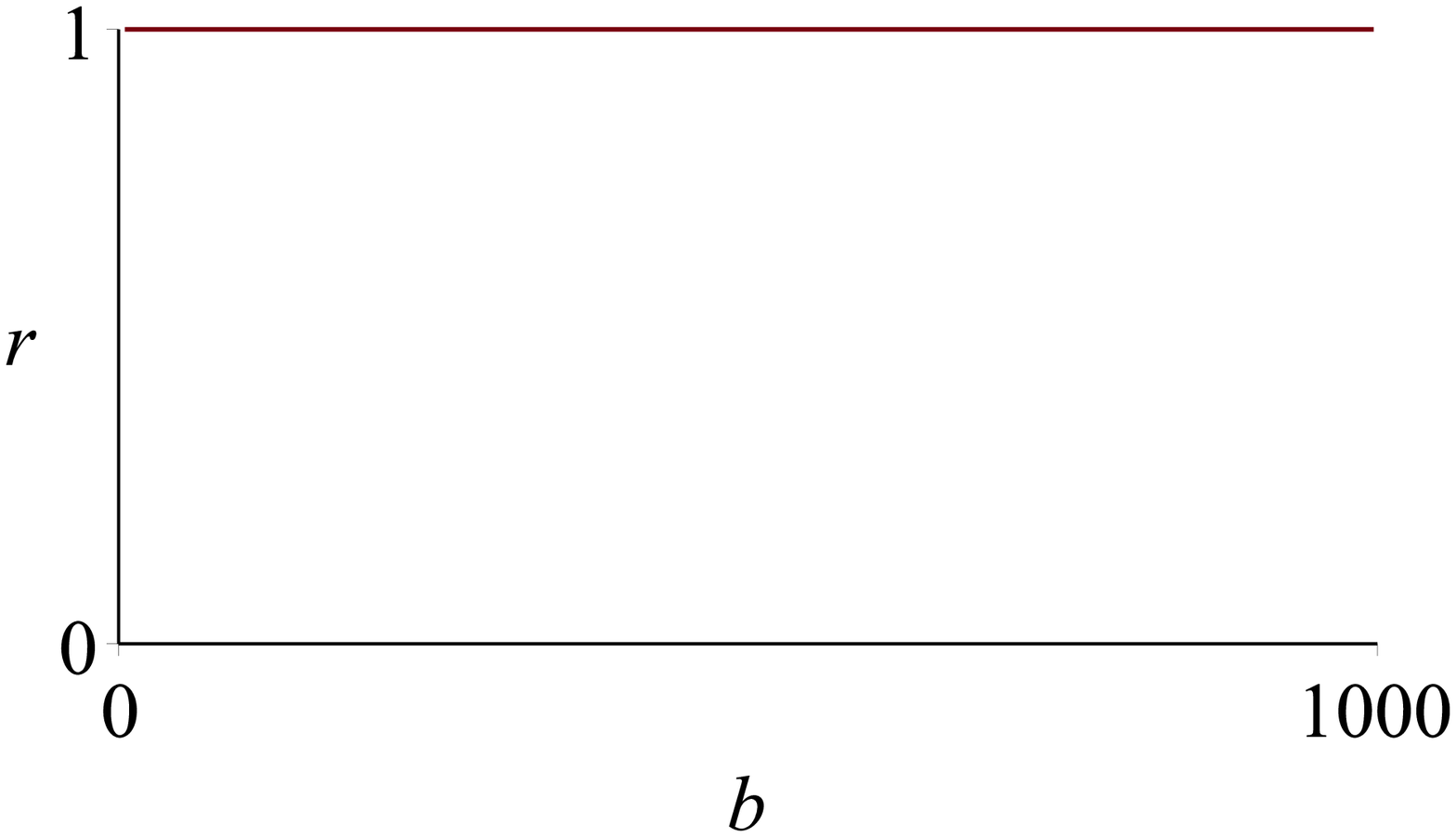} &  &
\includegraphics[height=1.2in]{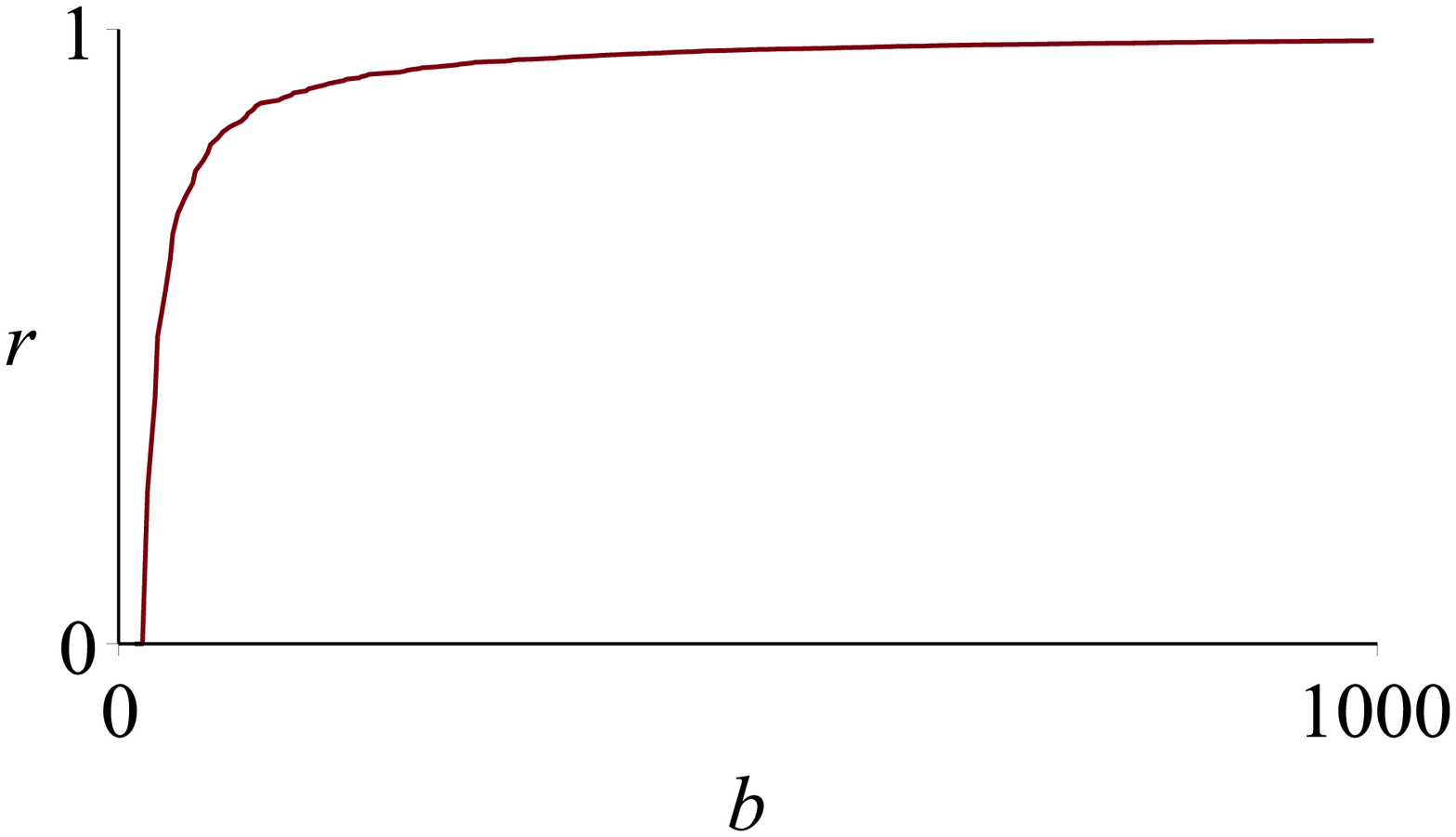}\\
k=3,\, p=3 &  & k=3,\, p=11\\
\includegraphics[height=1.2in]{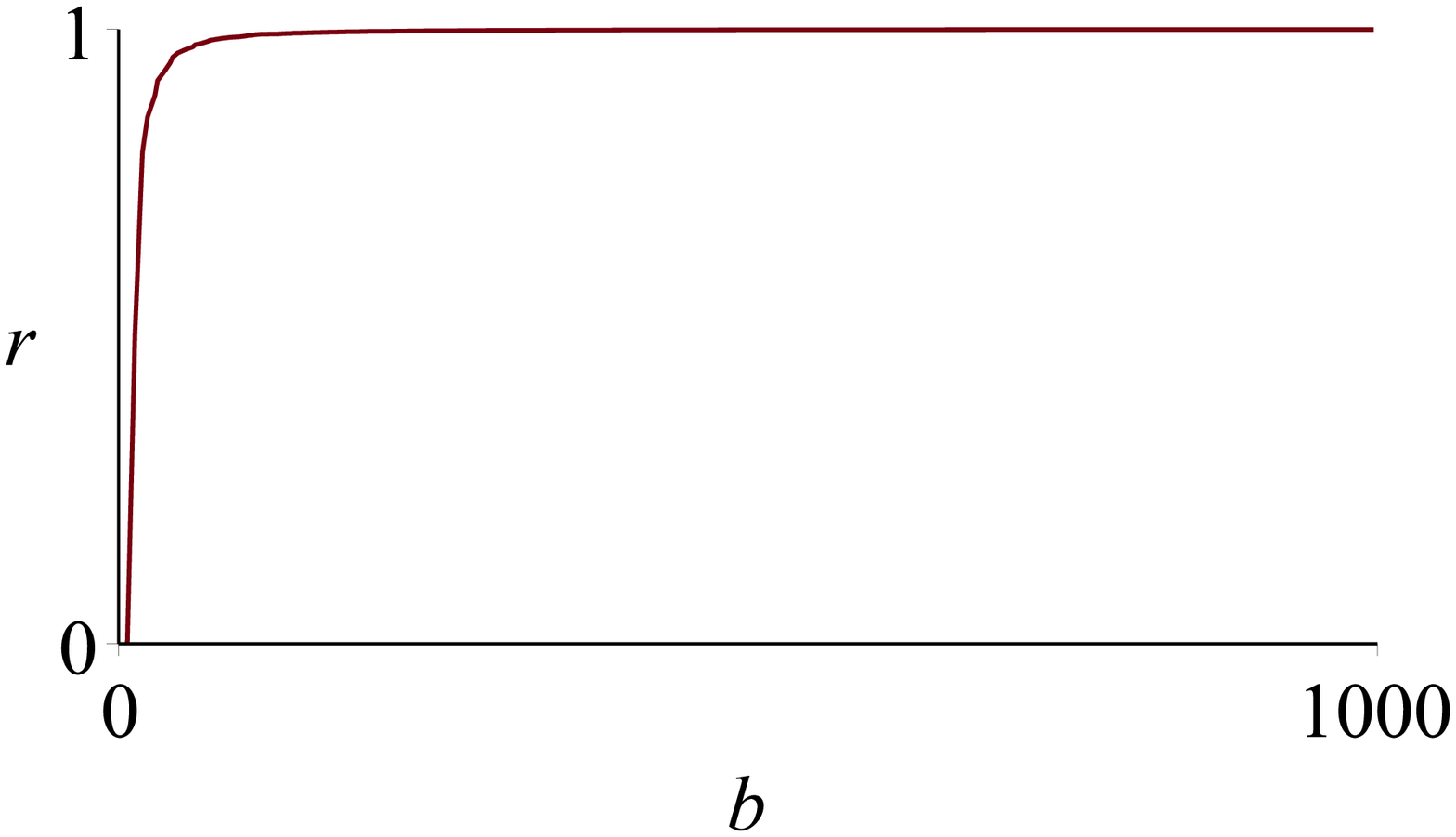}\hspace{20pt} &  &
\includegraphics[height=1.2in]{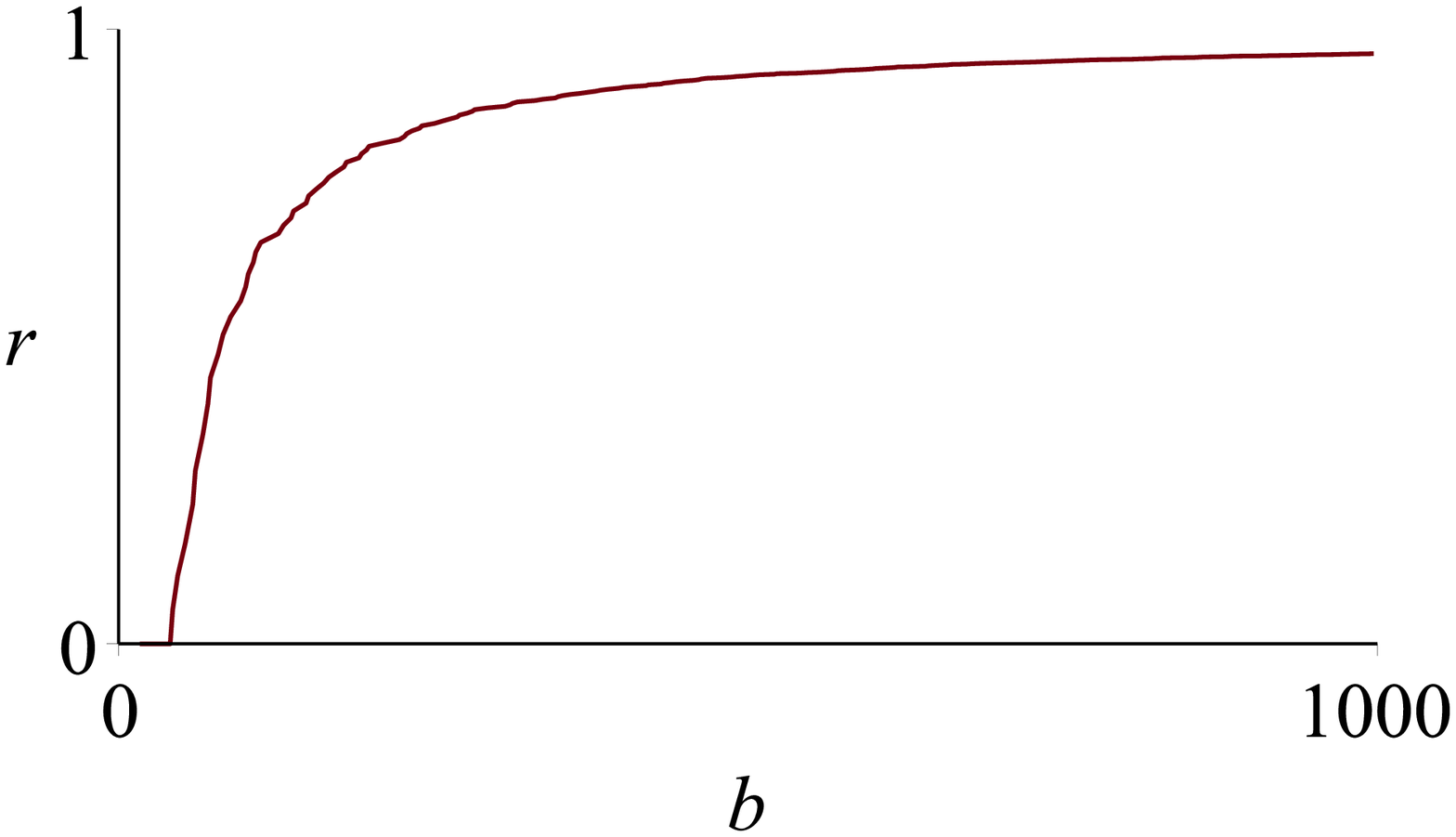}\\
&  &
\end{array}
\]


\section{Proof}

\label{sec:proof} In this section, we prove the main results
(Theorems~\ref{thm:alpha},~\ref{thm:beta},~\ref{thm:gamma},~\ref{thm:delta},~\ref{thm:epsilon}
and~\ref{thm:condition_delta}). We will first prove the general lower bound
$\varepsilon^{\pm}$ (Theorem~\ref{thm:epsilon}). Then we will prove the other
four special lower bounds $\alpha^{\pm},\,\beta^{\pm},\,\gamma^{\pm}$
 and $\delta^{-}$
(Theorems~\ref{thm:alpha},~\ref{thm:beta},~\ref{thm:gamma} and~\ref{thm:delta}) as certain
restrictions of Theorem~\ref{thm:epsilon}. After that, we will prove the
sufficient condition on $g(\Psi_{n})$ (Theorem~\ref{thm:condition_delta}). In
order to simplify the presentation of the proof, we introduce some notations.

\begin{notation}%
\[%
\begin{array}
[c]{rlll@{\;}lll}%
n_{r}=p_{1}\cdots p_{r} &  &  & u(A)=\min A &  &  & l^{\pm}%
(B)=\displaystyle\sum_{d\in B}\pm \mu\left(n/d\right) \ d
\end{array}
\]

\end{notation}

\subsection{Proof of General bound $\varepsilon^{\pm}$
(Theorem~\ref{thm:epsilon})}

We divide the proof into several lemmas.
\begin{notation}
Let
\[
F_{C}:=\prod_{c\in C}(x^{c}-1)
\]
and $F(C)=1$ if $C=\emptyset$.
\end{notation}

\begin{lemma}
\label{lem:condition_polynomial} We have that $\displaystyle \frac{F_{B^{\pm}%
}}{F_{B^{\mp}}}$ is a polynomial if
\begin{align*}
\mathcal{C}^{\pm}(B)  &  = \text{true}\\
B  &  \subset\underline{n}%
\end{align*}

\end{lemma}

\begin{proof}
Let $C\subset\underline{n}$. Consider the following equalities.
\[
F_{C}=\prod_{c\in C}(x^{c}-1)=\prod_{c\in C}\prod_{d|c}\Phi_{d}=\prod
_{d\in\underline{n}}\prod_{\substack{c\in C\\d|c}}\Phi_{d}=\prod
_{d\in\underline{n}}\Phi_{d}^{\#\{c\in C\;:\;d|c\}}=\prod_{d\in\underline{n}%
}\Phi_{d}^{\#(C\cap\overline{d})}%
\]
Thus%
\[
\frac{F_{B^{\pm}}}{F_{B^{\mp}}}=\ \frac{\prod\limits_{d\in\underline{n}}%
\Phi_{d}^{\#(B^{\pm}\cap\overline{d})}}{\prod\limits_{d\in\underline{n}}%
\Phi_{d}^{\#(B^{\mp}\cap\overline{d})}}=\ \prod_{d\in\underline{n}}\Phi
_{d}^{\#(B^{\pm}\cap\overline{d})-\#(B^{\mp}\cap\overline{d})}%
\]
Note that for $d\in\underline{n}\setminus\underline{B}$, we have $\#(B^{+}%
\cap\overline{d})=0$ and $\#(B^{-}\cap\overline{d})=0$. Thus,
\[
\frac{F_{B^{\pm}}}{F_{B^{\mp}}}=\ \prod_{d\in\underline{B}}\Phi_{d}%
^{\#(B^{\pm}\cap\overline{d})-\#(B^{\mp}\cap\overline{d})}%
\]
Recall $\mathcal{C}^{\pm}(B)\iff\forall d\in\underline{B}\;\;\#(B^{\pm}%
\cap\overline{d})\;\geq\;\#(B^{\mp}\cap\overline{d})$. Therefore,
$\frac{F_{B^{\pm}}}{F_{B^{\mp}}}$ is a polynomial.
\end{proof}

\begin{lemma}
\label{thm:codulo}We have%
\[
P^{\pm}\equiv_{x^{u(A)+1}}\pm%
\begin{cases}
-(-1)^{|A|}G^{\pm}-x^{u\left(  A\right)  } & \text{if\ }u\left(  A\right)  \in
A^{\pm}\\
-(-1)^{|A|}G^{\pm}+x^{u\left(  A\right)  } & \text{if\ }u\left(  A\right)  \in
A^{\mp}%
\end{cases}
\]
where%
\begin{align*}
P^{\pm}  &  =\frac{F_{\underline{n}^{\pm}}\ F_{\{n\}^{\mp}}}{F_{\underline{n}%
^{\mp}}}\\
G^{\pm}  &  =\frac{F_{B^{\pm}}}{F_{B^{\mp}}}\\
A\uplus B  &  =\underline{n}\setminus\{n\}\\
A  &  \neq\emptyset\\
\mathcal{C}^{\pm}(B)  &  =\text{true}\\
|B|  &  =\#\{b\in B\}
\end{align*}
and the notation $ \square \equiv_{x^{u(A)+1}} \triangle $ stands for $  x^{u(A)+1}| \square-\triangle $.

\end{lemma}

\begin{proof}
For simplicity, in the rest of this proof we will use $u$ instead of $u(A)$.
Since $A\neq\emptyset$, $u(A)$ is defined. Note%
\begin{align*}
F_{\underline{n}^{\mp}}\ P^{\pm}  &  =F_{\underline{n}^{\pm}}\ F_{\{n\}^{\mp}%
}\\
F_{\{n\}^{\mp}}\ F_{A^{\mp}}\ F_{B^{\mp}}\ P^{\pm}  &  =F_{A^{\pm}}%
\ F_{B^{\pm}}\ F_{\{n\}^{\pm}}\ F_{\{n\}^{\mp}}%
\end{align*}

\begin{itemize}
\item[Case:] $u\in A^{\pm}$. Since $u=\min A$ we have%
\begin{align*}
F_{A^{\pm}\setminus\{u\}}  &  \equiv_{x^{u+1}}(-1)^{|A^{\pm}|-1}\\
F_{A^{\mp}}  &  \equiv_{x^{u+1}}(-1)^{|A^{\mp}|}\\
F_{\{n\}^{\pm}}  &  \equiv_{x^{u+1}}(-1)^{|\{n\}^{\pm}|}\\
F_{\{n\}^{\mp}}  &  \equiv_{x^{u+1}}(-1)^{|\{n\}^{\mp}|}%
\end{align*}
Thus
\[
F_{B^{\mp}}\ P^{\pm}\equiv_{x^{u+1}}(-1)^{|A|-1+|\{n\}^{\pm}|}(x^{u}%
-1)\ F_{B^{\pm}}%
\]
Since $\mathcal{C}^{\pm}(B)$, by Lemma~\ref{lem:condition_polynomial} we have
\[
F_{B^{\mp}}\ P^{\pm}\equiv_{x^{u+1}}(-1)^{|A|-1+|\{n\}^{\pm}|}(x^{u}%
-1)\ F_{B^{\mp}}\ G^{\pm}%
\]
Note that $0$ is the only root of $x^{u+1}$ and $F_{B^{\mp}}(0)=(-1)^{|B^{\mp
}|}$. Hence $\gcd(F_{B^{\mp}},x^{u+1})=1$. Thus we can cancel $F_{B^{\mp}}$
from both sides, obtaining
\begin{align*}
P^{\pm}  &  \equiv_{x^{u+1}}(-1)^{|A|-1+|\{n\}^{\pm}|}(x^{u}-1)G^{\pm}\\
&  \equiv_{x^{u+1}}-(-1)^{|A|-1+|\{n\}^{\pm}|}G^{\pm}+(-1)^{|A|-1+|\{n\}^{\pm
}|}x^{u}G^{\pm}%
\end{align*}
Note that $G^{\pm}(0)=(-1)^{|B|}$. Thus we have%
\begin{align*}
P^{\pm}  &  \equiv_{x^{u+1}}-(-1)^{|A|-1+|\{n\}^{\pm}|}G^{\pm}%
+(-1)^{|A|-1+|\{n\}^{\pm}|+\left\vert B\right\vert }x^{u}\\
&  \equiv_{x^{u+1}}\left(  -1\right)  ^{-1+|\{n\}^{\pm}|}\left(
-(-1)^{|A|}G^{\pm}+(-1)^{|A|+|B|}x^{u}\right) \\
&  \equiv_{x^{u+1}}\pm\left(  -(-1)^{|A|}G^{\pm}+(-1)^{2^{k}-1}x^{u}\right) \\
&  \equiv_{x^{u+1}}\pm\left(  -(-1)^{|A|}G^{\pm}-x^{u}\right)
\end{align*}
which proves the lemma.

\item[Case:] $u\in A^{\mp}$. Since $u=\min A$ we have%
\begin{align*}
F_{A^{\pm}}  &  \equiv_{x^{u+1}}(-1)^{|A^{\pm}|}\\
F_{A^{\mp}\setminus\{u\}}  &  \equiv_{x^{u+1}}(-1)^{|A^{\mp}|-1}\\
F_{\{n\}^{\pm}}  &  \equiv_{x^{u+1}}(-1)^{|\{n\}^{\pm}|}\\
F_{\{n\}^{\mp}}  &  \equiv_{x^{u+1}}(-1)^{|\{n\}^{\mp}|}%
\end{align*}
Thus
\[
(x^{u}-1)F_{B^{\mp}}\cdot P^{\pm}\equiv_{x^{u+1}}(-1)^{|A|-1+|\{n\}^{\pm}%
|}F_{B^{\pm}}%
\]
Since $\mathcal{C}^{\pm}(B)$, by Lemma~\ref{lem:condition_polynomial} we have
\[
(x^{u}-1)F_{B^{\mp}}\cdot P^{\pm}\equiv_{x^{u+1}}(-1)^{|A|-1+|\{n\}^{\pm}%
|}F_{B^{\mp}}\cdot G^{\pm}%
\]
Note that $0$ is the only root of $x^{u+1}$ and $F_{B^{\mp}}(0)=(-1)^{|B^{\mp
}|}$. Hence $\gcd(F_{B^{\mp}},x^{u+1})=1$. Thus we can cancel $F_{B^{\mp}}$
from both sides, obtaining
\[
(x^{u}-1)\ P^{\pm}\equiv_{x^{u+1}}(-1)^{|A|-1+|\{n\}^{\pm}|}G^{\pm}%
\]
Multiplying both sides by $(x^{u}+1)$, we have
\begin{align*}
(x^{u}+1)(x^{u}-1)P^{\pm}  &  \equiv_{x^{u+1}}(x^{u}+1)(-1)^{|A|-1+|\{n\}^{\pm
}|}G^{\pm}\\
(x^{2u}-1)P^{\pm}  &  \equiv_{x^{u+1}}(-1)^{|A|-1+|\{n\}^{\pm}|}G^{\pm
}+(-1)^{|A|-1+|\{n\}^{\pm}|}x^{u}G^{\pm}\\
-P^{\pm}  &  \equiv_{x^{u+1}}(-1)^{|A|-1+|\{n\}^{\pm}|}G^{\pm}%
+(-1)^{|A|-1+|\{n\}^{\pm}|}x^{u}G^{\pm}\\
P^{\pm}  &  \equiv_{x^{u+1}}-(-1)^{|A|-1+|\{n\}^{\pm}|}G^{\pm}%
-(-1)^{|A|-1+|\{n\}^{\pm}|}x^{u}G^{\pm}%
\end{align*}
Note that $G^{\pm}(0)=(-1)^{|B|}$. Thus we have%
\begin{align*}
P^{\pm}  &  \equiv_{x^{u+1}}-(-1)^{|A|-1+|\{n\}^{\pm}|}G^{\pm}%
-(-1)^{|A|-1+|\{n\}^{\pm}|+|B|}x^{u}\\
&  \equiv_{x^{u+1}}\left(  -1\right)  ^{-1+|\{n\}^{\pm}|}\left(
-(-1)^{|A|}G^{\pm}-(-1)^{|A|+|B|}x^{u}\right) \\
&  \equiv_{x^{u+1}}\pm\left(  -(-1)^{|A|}G^{\pm}-(-1)^{2^{k-1}}x^{u}\right) \\
&  \equiv_{x^{u+1}}\pm\left(  -(-1)^{|A|}G^{\pm}+x^{u}\right)
\end{align*}
which proves the lemma.
\end{itemize}
\end{proof}

\begin{proof}
[Proof of Theorem~\ref{thm:epsilon}-(1)]

Using the same notation as in Lemma~\ref{thm:codulo}, note
\[
P^{+}=\Phi_{n}%
\]
Let $A$ and $B$ be such that $A\uplus B=\underline{n}\setminus\{n\}$,
$A\neq\emptyset$, and $\mathcal{C}^{+}(B)$. By Lemma~\ref{thm:codulo}, we have%

\[
\Phi_{n}=+%
\begin{cases}
-(-1)^{|A|}G^{+}-x^{u\left(  A\right)  }+x^{u\left(  A\right)  +1}H & \text{if
}u\left(  A\right)  \in A^{+}\\
-(-1)^{|A|}G^{+}+x^{u\left(  A\right)  }+x^{u\left(  A\right)  +1}H & \text{if
}u\left(  A\right)  \in A^{-}%
\end{cases}
\]
for some polynomial $H$. Note%
\[
\deg G^{+}=\deg\left(  \frac{F_{B^{+}}}{F_{B^{-}}}\right)  =\sum_{d\in B^{+}%
}d-\sum_{d\in B^{-}}d=\sum_{d\in B}\mu\left(  n/d\right)  d=l^{+}(B)
\]
If $u(A)\leq l^{+}(B)$, then clearly
\[
g(\Phi_{n})\geq u(A)-l^{+}(B)
\]
If $u(A)>l^{+}(B)$, then $x^{l^{+}(B)}$ and $x^{u(A)}$ appear in $\Phi_{n}$,
so we have%
\[
g(\Phi_{n})\geq u(A)-l^{+}(B)
\]
Thus%
\[
g(\Phi_{n})\geq\max_{\substack{A\uplus B=\underline{n}\setminus\{n\}\\A\neq
\emptyset\\\mathcal{C}^{+}(B)}}u(A)-l^{+}(B)=\varepsilon^{+}(n)
\]
The theorem has been proved.
\end{proof}

\begin{proof}
[Proof of Theorem~\ref{thm:epsilon}-(2)]Using the same notation as in
Lemma~\ref{thm:codulo}, note%
\[
P^{-}=\Psi_{n}%
\]
Let $A$ and $B$ be such that $A\uplus B=\underline{n}\setminus\{n\}$,
$A\neq\emptyset$, and $\mathcal{C}^{-}(B)$. By Lemma~\ref{thm:codulo}, we have%
\[
\Psi_{n}=-%
\begin{cases}
-(-1)^{|A|}G^{-}-x^{u\left(  A\right)  }+x^{u\left(  A\right)  +1}H &
\text{if\ }u\left(  A\right)  \in A^{-}\\
-(-1)^{|A|}G^{-}+x^{u\left(  A\right)  }+x^{u\left(  A\right)  +1}H &
\text{if\ }u\left(  A\right)  \in A^{+}%
\end{cases}
\]
for some polynomial $H$. Note%
\[
\deg G^{-}=\deg\left(  \frac{F_{B^{-}}}{F_{B^{+}}}\right)  =\sum_{d\in B^{-}%
}d-\sum_{d\in B^{+}}d=\sum_{d\in B}-\mu\left(  n/d\right)  d=l^{-}(B)
\]
If $u(A)\leq l^{-}(B)$, then clearly
\[
g(\Psi_{n})\geq u(A)-l^{-}(B)
\]
If $u(A)>l^{-}(B)$, then $x^{l^{-}(B)}$ and $x^{u(A)}$ appear in $\Psi_{n}$,
so we have%
\[
g(\Psi_{n})\geq u(A)-l^{-}(B)
\]
Thus
\[
g(\Psi_{n})\geq\max_{\substack{A\uplus B=\underline{n}\setminus\{n\}\\A\neq
\emptyset\\\mathcal{C}^{-}(B)}}u(A)-l^{-}(B)=\varepsilon^{-}(n)
\]
The theorem has been proved.
\end{proof}

\subsection{Proof of Special bounds $\alpha^{\pm},\ \beta^{\pm}$ and
$\gamma^{\pm}$ (Theorems~\ref{thm:alpha}, \ref{thm:beta} and \ref{thm:gamma})}

We restrict the choice of $B$ as mentioned in Section~\ref{sec:results}. Note
that the restrictions are very similar. To deal with them at the same time, we
will use the following uniform notation%
\[
\Omega_{jr}=\left\{  c\in\underline{n_{j}}:\omega\left(  c\right)  <r\right\}
\]
Note that $B$ for $\alpha^{\pm},\ \beta^{\pm}$ and $\gamma^{\pm}$ can be
compactly written as $B=\Omega_{r-1,r},\ B=\Omega_{rr}$ and $B=\Omega_{kr}$
respectively. In the following three lemmas, we will show that $\mathcal{C}%
^{\pm}\left(  \Omega_{jr}\right)  $ holds.

\begin{lemma}
\label{lem:num_sum_binom} We have, for $s\in\left\{  +,-\right\}  $, that%
\[
\#\left(  \Omega_{jr}^{s}\cap\overline{d}\right)  =\sum_{_{\substack{0\leq
i<r-\omega\left(  d\right)  \\\rho(i)=s\rho(k-\omega\left(  d\right)  )}%
}}\binom{j-\omega\left(  d\right)  }{i}%
\]
for $1\leq r<k$, $r-1\leq j\leq k$ and $d\in\underline{\Omega_{jr}}.$
\end{lemma}

\begin{proof}
Note%
\[%
\begin{array}
[c]{rlll}%
\#\left(  \Omega_{jr}^{s}\cap\overline{d}\right)  & =\#\big \{c\in
\underline{n_{j}} & \;:\;\omega\left(  c\right)  <r, & \mu
(n/c)=s,\,d\ |\ c\ \big \}\\
& =\#\big \{ld\in\underline{n_{j}} & \;:\;\omega\left(  ld\right)  <r, &
\mu(n/\left(  ld\right)  )=s\big \}\\
& =\#\big \{l\in\underline{n_{j}/d} & \;:\;\omega(l)<r-\omega\left(  d\right)
, & \mu(l)=s\mu\left(  n/d\right)  \big \}
\end{array}
\]
Note%
\[
s\mu\left(  n/d\right)  =s\rho\left(  k-\omega\left(  d\right)  \right)
\]
Thus%
\begin{align*}
\#\left(  \Omega_{jr}^{s}\cap\overline{d}\right)   &  =\#%
{\displaystyle\biguplus\limits_{_{\substack{0\leq i<r-\omega\left(  d\right)
\\\rho(i)=s\rho(k-\omega\left(  d\right)  )}}}}
\left\{  l\in\underline{n_{j}/d}:\omega(l)=i\right\} \\
&  =\sum_{_{\substack{0\leq i<r-\omega\left(  d\right)  \\\rho(i)=s\rho
(k-\omega\left(  d\right)  )}}}\#\left\{  l\in\underline{n_{j}/d}%
:\omega(l)=i\right\} \\
&  =\sum_{_{\substack{0\leq i<r-\omega\left(  d\right)  \\\rho(i)=s\rho
(k-\omega\left(  d\right)  )}}}\binom{\omega\left(  n_{j}/d\right)  }{i}\\
&  =\sum_{_{\substack{0\leq i<r-\omega\left(  d\right)  \\\rho(i)=s\rho
(k-\omega\left(  d\right)  )}}}\binom{j-\omega\left(  d\right)  }{i}%
\end{align*}
which proves the lemma.
\end{proof}

\begin{lemma}
[Telescoping sum]\label{thm:telescope} We have%
\[
\sum_{0\leq i\leq u}\rho(i)\binom{t}{i}=\left\{
\begin{array}
[c]{ll}%
\rho(u)\binom{t-1}{u} & \text{if }t\geq1\\
1 & \text{if }t=0
\end{array}
\right.
\]

\end{lemma}

\begin{proof}
When $t\geq1$, we have%
\begin{align*}
\sum_{0\leq i\leq u}\rho(i)\binom{t}{i}  &  =\sum_{0\leq i\leq u}\rho
(i)\binom{t-1}{i-1}+\sum_{0\leq i\leq u}\rho(i)\binom{t-1}{i}\\
&  =-\sum_{-1\leq i\leq u-1}\rho(i)\binom{t-1}{i}+\sum_{0\leq i\leq u}%
\rho(i)\binom{t-1}{i}\\
&  =-\rho\left(  -1\right)  \binom{t-1}{-1}+\rho\left(  u\right)  \binom
{t-1}{u}\\
&  =\rho(u)\binom{t-1}{u}%
\end{align*}
When $t=0$, we have%
\[
\sum_{0\leq i\leq u}\rho(i)\binom{t}{i}=\rho\left(  0\right)  \binom{0}%
{0}+\sum_{1\leq i\leq u}\rho(i)\binom{0}{i}=1+0=1
\]

\end{proof}

\begin{lemma}
\label{lem:C_true} We have $\mathcal{C}^{\pm}(\Omega_{jr})$ for $1\leq r<k$,
$\rho(k-r)=\mp1$ and $r-1\leq j\leq k.$
\end{lemma}

\begin{proof}
Recall%
\[
\mathcal{C}^{\pm}(\Omega_{jr})\ \ \iff\ \ \forall d\in\underline{\Omega_{jr}%
}\;\;\#\left(  \Omega_{jr}^{\pm}\cap\overline{d}\right)  \;\geq\;\#\left(
\Omega_{jr}^{\mp}\cap\overline{d}\right)
\]
Note%
\begin{align*}
&  \#\left(  \Omega_{jr}^{\pm}\cap\overline{d}\right)  -\#\left(  \Omega
_{jr}^{\mp}\cap\overline{d}\right) \\
&  =\sum_{_{\substack{0\leq i<r-\omega\left(  d\right)  \\\rho(i)=\pm
\rho(k-\omega\left(  d\right)  )}}}\binom{j-\omega\left(  d\right)  }{i}%
-\sum_{_{\substack{0\leq i<r-\omega\left(  d\right)  \\\rho(i)=\mp
\rho(k-\omega\left(  d\right)  )}}}\binom{j-\omega\left(  d\right)  }%
{i}\ \ \text{by Lemma~\ref{lem:num_sum_binom}}\\
&  =\sum_{_{\substack{0\leq i<r-\omega\left(  d\right)  \\\rho(i)=-\rho
(k-r)\rho(k-\omega\left(  d\right)  )}}}\binom{j-\omega\left(  d\right)  }%
{i}-\sum_{_{\substack{0\leq i<r-\omega\left(  d\right)  \\\rho(i)=+\rho
(k-r)\rho(k-\omega\left(  d\right)  )}}}\binom{j-\omega\left(  d\right)  }%
{i}\ \ \text{since }\rho\left(  k-r\right)  =\mp1\\
&  =\sum_{_{\substack{0\leq i<r-\omega\left(  d\right)  \\\rho(i)=-\rho\left(
r-\omega\left(  d\right)  \right)  }}}\binom{j-\omega\left(  d\right)  }%
{i}-\sum_{_{\substack{0\leq i<r-\omega\left(  d\right)  \\\rho(i)=+\rho\left(
r-\omega\left(  d\right)  \right)  }}}\binom{j-\omega\left(  d\right)  }%
{i}\ \ \\
&  =\sum_{0\leq i<r-\omega\left(  d\right)  }-\rho\left(  r-\omega\left(
d\right)  \right)  \rho\left(  i\right)  \binom{j-\omega\left(  d\right)  }%
{i}\\
&  =-\rho\left(  r-\omega\left(  d\right)  \right)  \sum_{0\leq i<r-\omega
\left(  d\right)  }\rho\left(  i\right)  \binom{j-\omega\left(  d\right)  }%
{i}\\
&  =-\rho\left(  r-\omega\left(  d\right)  \right)  \left\{
\begin{array}
[c]{ll}%
\rho\left(  r-\omega\left(  d\right)  -1\right)  \binom{j-\omega\left(
d\right)  -1}{r-\omega\left(  d\right)  -1} & \text{if }j-\omega\left(
d\right)  \geq1\\
1 & \text{if }j-\omega\left(  d\right)  =0
\end{array}
\right.  \ \ \text{by Lemma~\ref{thm:telescope} }\\
&  =\left\{
\begin{array}
[c]{ll}%
-\rho\left(  r-\omega\left(  d\right)  \right)  \rho\left(  r-\omega\left(
d\right)  -1\right)  \binom{j-\omega\left(  d\right)  -1}{r-\omega\left(
d\right)  -1} & \text{if }j-\omega\left(  d\right)  \geq1\\
-\rho\left(  r-\omega\left(  d\right)  \right)  & \text{if }j-\omega\left(
d\right)  =0
\end{array}
\right. \\
&  =\left\{
\begin{array}
[c]{ll}%
\binom{j-\omega\left(  d\right)  -1}{r-\omega\left(  d\right)  -1} & \text{if
}j-\omega\left(  d\right)  \geq1\\
-\rho\left(  r-\omega\left(  d\right)  \right)  & \text{if }j-\omega\left(
d\right)  =0
\end{array}
\right.
\end{align*}
Consider the case $j-\omega\left(  d\right)  =0$: Since $r-1\leq
j=\omega\left(  d\right)  \leq r-1$, we have $\omega\left(  d\right)  =r-1.$
Therefore we have%

\begin{align*}
\#\left(  \Omega_{jr}^{\pm}\cap\overline{d}\right)  -\#\left(  \Omega
_{jr}^{\mp}\cap\overline{d}\right)   &  =\left\{
\begin{array}
[c]{ll}%
\binom{j-\omega\left(  d\right)  -1}{r-\omega\left(  d\right)  -1} & \text{if
}j-\omega\left(  d\right)  \geq1\\
-\rho\left(  1\right)  & \text{if }j-\omega\left(  d\right)  =0
\end{array}
\right. \\
&  =\left\{
\begin{array}
[c]{ll}%
\binom{j-\omega\left(  d\right)  -1}{r-\omega\left(  d\right)  -1} & \text{if
}j-\omega\left(  d\right)  \geq1\\
1 & \text{if }j-\omega\left(  d\right)  =0
\end{array}
\right. \\
&  \geq0
\end{align*}
\noindent which proves the lemma.
\end{proof}

\begin{lemma}
\label{lem:ul}We have, for $r-1\leq j\leq k$,%
\[
\varepsilon^{\pm}(n)\geq\max_{\substack{1\leq r<k\\\rho(k-r)=\mp1}}\left(
\min\left\{  p_{j+1},n_{r}\right\}  -\sum_{\substack{d|n_{j}\\\omega\left(
d\right)  <r}}\pm\mu\left(  n/d\right)  \ d\right)
\]
where $p_{k+1}$ is viewed as $\infty.$
\end{lemma}

\begin{proof}
Note%
\begin{align*}
\varepsilon^{\pm}(n)  &  =\max_{\substack{A\uplus B=\underline{n}%
\setminus\{n\}\\A\neq\emptyset\\\mathcal{C}^{\pm}(B)}}u(A)-l^{\pm}(B)\\
&  \geq\max_{\substack{A\uplus B=\underline{n}\setminus\{n\}\\A\neq
\emptyset\\\mathcal{C}^{\pm}(B)\\1\leq r<k\\\rho(k-r)=\mp1\\B=\Omega_{jr}%
}}u(A)-l^{\pm}(B)\ \ \ \ \ \ \ \text{by restricting the choice of }B~\text{to
}\Omega_{jr}\\
&  =\max_{_{\substack{1\leq r<k\\\rho(k-r)=\mp1\\\mathcal{C}^{\pm}(\Omega
_{jr})}}}u\left(  \underline{n}\setminus\{n\}\setminus\Omega_{jr}\right)
-l^{\pm}\left(  \Omega_{jr}\right) \\
&  =\max_{\substack{1\leq r<k\\\rho(k-r)=\mp1}}u\left(  \underline{n}%
\setminus\{n\}\setminus\Omega_{jr}\right)  -l^{\pm}\left(  \Omega_{jr}\right)
\ \text{by Lemma \ref{lem:C_true}}%
\end{align*}
Note%
\begin{align*}
u\left(  \underline{n}\setminus\{n\}\setminus\Omega_{jr}\right)   &
=\min\left(  \underline{n}\setminus\{n\}\setminus\Omega_{jr}\right) \\
&  =\min\left(  \underline{n}\setminus\{n\}\setminus\left\{  c:c|n_{j}%
,\omega\left(  c\right)  <r\right\}  \right) \\
&  =\min\left\{  c:c|n,c\neq n\,\text{and}\ \left(  c\nmid n_{j}%
\ \text{or\ }\omega\left(  c\right)  \geq r\right)  \right\} \\
&  =\min\big(\min\left\{  c:c|n,c\neq n\,\text{and}\ c\nmid n_{j}\right\}
,\min\left\{  c:c|n,c\neq n\,\text{and}\ \omega\left(  c\right)  \geq
r\right\}  \big)\\
&  =\min\left(  \min\left\{  c:c|n,c\neq n\,\text{and}\ \underset{i\geq
j+1}{\exists}p_{i}|c\right\}  ,n_{r}\right) \\
&  =\min\left\{  p_{j+1},n_{r}\right\}
\end{align*}
Note%
\[
l^{\pm}\left(  \Omega_{jr}\right)  =\sum_{d\in\Omega_{jr}}\pm\mu\left(
n/d\right)  \ d=\sum_{\substack{d\in\underline{n_{j}}\\\omega\left(  d\right)
<r}}\pm\mu\left(  n/d\right)  \ d
\]
Hence%
\[
\varepsilon^{\pm}(n)\geq\max_{\substack{1\leq r<k\\\rho(k-r)=\mp1}}\left(
\min\left\{  p_{j+1},n_{r}\right\}  -\sum_{\substack{d\in\underline{n_{j}%
}\\\omega\left(  d\right)  <r}}\pm\mu\left(  n/d\right)  \ d\right)
\]

\end{proof}

\begin{lemma}
\label{lem:phi}We have%
\[
\varphi\left(  n_{r}\right)  =\sum_{d|n_{r}}\mu\left(  n_{r}/d\right)  \ d
\]

\end{lemma}

\begin{proof}
Note%
\[
\varphi\left(  n_{r}\right)  =\left(  p_{1}-1\right)  \cdots\left(
p_{r}-1\right)  =\left(  -1\right)  ^{r}\left(  1-p_{1}\right)  \cdots\left(
1-p_{r}\right)  =\left(  -1\right)  ^{r}\sum_{d|n_{r}}\mu\left(  d\right)
\ d=\sum_{d|n_{r}}\mu\left(  n_{r}/d\right)  \ d
\]

\end{proof}

\begin{proof}
[Proof of Theorem \ref{thm:alpha}]$\,$We set $j=r-1.$ Note%
\[
\varepsilon^{\pm}(n)\geq\max_{\substack{1\leq r<k\\\rho(k-r)=\mp1}}\left(
\min\left\{  p_{r-1+1},n_{r}\right\}  -\sum_{\substack{d|n_{r-1}%
\\\omega\left(  d\right)  <r}}\pm\mu\left(  n/d\right)  \ d\right)
\ \ \ \text{by\ Lemma \ref{lem:ul}}%
\]
Note that%
\[
\min\left\{  p_{r-1+1},n_{r}\right\}  =\min\left\{  p_{r},n_{r}\right\}
=p_{r}%
\]
Not that%
\begin{align*}
\sum_{\substack{d|n_{r-1}\\\omega\left(  d\right)  <r}}\pm\mu\left(
n/d\right)  \ d  &  =\sum_{d|n_{r-1}}\pm\mu\left(  n/d\right)  \ d\\
&  =\sum_{d|n_{r-1}}\pm\mu\left(  n/n_{r-1}\right)  \mu\left(  n_{r-1}%
/d\right)  \ d\\
&  =\sum_{d|n_{r-1}}\pm1\cdot\pm1\mu\left(  n_{r-1}/d\right)  \ d\\
&  =\sum_{d|n_{r-1}}\mu\left(  n_{r-1}/d\right)  \ d\\
&  =\varphi\left(  n_{r-1}\right)  \ \ \ \text{by Lemma \ref{lem:phi}}%
\end{align*}
Thus%
\[
\varepsilon^{\pm}(n)\geq\max_{\substack{1\leq r<k\\\rho(k-r)=\mp1}}\left(
p_{r}-\varphi\left(  n_{r-1}\right)  \right)  =\max_{\substack{1\leq
r<k\\\rho(k-r)=\mp1}}\left(  p_{r}-\varphi\left(  p_{1}\cdots p_{r-1}\right)
\right)  =\alpha^{\pm}(n)
\]
Hence%
\begin{align*}
g(\Phi_{n})  &  \geq\varepsilon^{+}(n)\geq\alpha^{+}(n)\\
g(\Psi_{n})  &  \geq\varepsilon^{-}(n)\geq\alpha^{-}(n)
\end{align*}
The theorem has been proved.
\end{proof}

\begin{proof}
[Proof of Theorem~\ref{thm:beta}]We set $j=r.$ Note
\[
\varepsilon^{\pm}(n)\geq\max_{\substack{1\leq r<k\\\rho(k-r)=\mp1}}\left(
\min\left\{  p_{r+1},n_{r}\right\}  -\sum_{\substack{d\in\underline{n_{r}%
}\\\omega\left(  d\right)  <r}}\pm\mu\left(  n/d\right)  \ d\right)
\ \ \ \text{by\ Lemma \ref{lem:ul}}%
\]
Note that
\begin{align*}
\sum_{\substack{d\in\underline{n_{r}}\\\omega\left(  d\right)  <r}}\pm
\mu\left(  n/d\right)  \ d  &  =\sum_{d\in\underline{n_{r}}\setminus\{n_{r}%
\}}\pm\mu\left(  n/n_{r}\right)  \mu\left(  n_{r}/d\right)  \ d\\
&  =\sum_{d\in\underline{n_{r}}\setminus\{n_{r}\}}\pm1\cdot\mp1\cdot\mu\left(
n_{r}/d\right)  \ d\\
&  =-\sum_{d\in\underline{n_{r}}\setminus\{n_{r}\}}\mu\left(  n_{r}/d\right)
\ d\\
&  =\mu\left(  n_{r}/n_{r}\right)  \ n_{r}-\sum_{d\in\underline{n_{r}}}%
\mu\left(  n_{r}/d\right)  \ d\\
&  =n_{r}-\sum_{d|n_{r}}\mu\left(  n_{r}/d\right)  \ d\\
&  =n_{r}-\varphi\left(  n_{r}\right)  \ \ \ \ \text{by Lemma \ref{lem:phi}}\\
&  =\psi(n_{r})
\end{align*}
Thus%
\[
\varepsilon^{\pm}(n)\geq\max_{\substack{1\leq r<k\\\rho(k-r)=\mp1}}\left(
\min\left\{  p_{r+1},n_{r}\right\}  -\psi(n_{r})\right)  =\max
_{\substack{1\leq r<k\\\rho(k-r)=\mp1}}\left(  \min\left\{  p_{r+1}%
,p_{1}\cdots p_{r}\right\}  -\psi(p_{1}\cdots p_{r})\right)  =\beta^{\pm}(n)
\]
Hence%
\begin{align*}
g(\Phi_{n})  &  \geq\varepsilon^{+}(n)\geq\beta^{+}(n)\\
g(\Psi_{n})  &  \geq\varepsilon^{-}(n)\geq\beta^{-}(n)
\end{align*}
The theorem has been proved.
\end{proof}

\begin{proof}
[Proof of Theorem~\ref{thm:gamma}]We set $j=k.$ Note%
\[
\varepsilon^{\pm}(n)\geq\max_{\substack{1\leq r<k\\\rho(k-r)=\mp1}}\left(
\min\left\{  p_{k+1},n_{r}\right\}  -\sum_{\substack{d\in\underline{n}%
\\\omega\left(  d\right)  <r}}\pm\mu\left(  n/d\right)  \ d\right)
\ \ \ \text{by\ Lemma \ref{lem:ul}}%
\]
Note%
\[
\min\left\{  p_{k+1},n_{r}\right\}  =\min\left\{  \infty,n_{r}\right\}  =n_{r}%
\]
Thus%
\[
\varepsilon^{\pm}(n)\geq\max_{\substack{1\leq r<k\\\rho(k-r)=\mp1}}\left(
n_{r}-\sum_{\substack{d\in\underline{n}\\\omega\left(  d\right)  <r}}\pm
\mu\left(  n/d\right)  \ d\right)  =\max_{\substack{1\leq r<k\\\rho(k-r)=\mp
1}}\left(  p_{1}\cdots p_{r}-\sum_{\substack{d|n\\\omega\left(  d\right)
<r}}\pm\mu\left(  n/d\right)  \ d\right)  =\gamma^{\pm}(n)
\]
Hence%
\begin{align*}
g(\Phi_{n})  &  \geq\varepsilon^{+}(n)\geq\gamma^{+}(n)\\
g(\Psi_{n})  &  \geq\varepsilon^{-}(n)\geq\gamma^{-}(n)
\end{align*}
The theorem has been proved.
\end{proof}

\subsection{Proof of Special bound $\delta^{-}$ (Theorem~\ref{thm:delta})}

It is possible to prove Theorem~\ref{thm:delta} in a similar way to the last
three theorems, by restricting $B$ as mentioned in Section~\ref{sec:results},
that is,%
\[
B=\left\{  d:d|p_{1}\cdots p_{k}\ \text{and }\omega\left(  d\right)
<k\ \text{and }d\neq p_{2}\cdots p_{k}\right\}  .
\]
However, it is simpler to prove it in a different way.

\begin{lemma}
\label{lem:middle_structure} We have
\[
\Psi_{n}(x)=H(x)\left(  x^{\frac{n}{p_{1}}}-1\right)
\]
where $H(x)=\Phi_{n_{k-1}}\left(  x^{\frac{n}{n_{k}}}\right)  \Phi_{n_{k-2}%
}\left(  x^{\frac{n}{n_{k-1}}}\right)  \cdots\Phi_{n_{1}}\left(  x^{\frac
{n}{n_{2}}}\right)  $.
\end{lemma}

\begin{proof}
Recall the well known property of cyclotomic polynomials
\[
\Phi_{np}(x)=\frac{\Phi_{n}(x^{p})}{\Phi_{n}(x)}%
\]
where $p$ is a prime and not a factor of $n$. In terms of the inverse
cyclotomic polynomial, it can be immediately restated as
\[
\Psi_{np}(x)=\Phi_{n}(x)\Psi_{n}(x^{p})
\]
Repeatedly applying the above equality on $\Psi_{n}(x)$, we have
\begin{align*}
\Psi_{n}(x)  &  =\Phi_{n_{k-1}}\left(  x^{\frac{n}{n_{k}}}\right)
\Psi_{n_{k-1}}\left(  x^{\frac{n}{n_{k-1}}}\right) \\
&  =\Phi_{n_{k-1}}\left(  x^{\frac{n}{n_{k}}}\right)  \Phi_{n_{k-2}}\left(
x^{\frac{n}{n_{k-1}}}\right)  \Psi_{n_{k-2}}\left(  x^{\frac{n}{n_{k-2}}%
}\right) \\
&  =\cdots\\
&  =\Phi_{n_{k-1}}\left(  x^{\frac{n}{n_{k}}}\right)  \Phi_{n_{k-2}}\left(
x^{\frac{n}{n_{k-1}}}\right)  \cdots\Phi_{n_{1}}\left(  x^{\frac{n}{n_{2}}%
}\right)  \Psi_{n_{1}}\left(  x^{\frac{n}{p_{1}}}\right) \\
&  =H(x)\;\Psi_{n_{1}}\left(  x^{\frac{n}{p_{1}}}\right)
\end{align*}
Recall that for a prime $p$, we have
\[
\Psi_{p}(x)=x-1
\]
Hence
\[
\Psi_{n}(x)=H(x)\left(  x^{\frac{n}{p_{1}}}-1\right)
\]

\end{proof}

\begin{proof}
[Proof of Theorem~\ref{thm:delta}]From Lemma~\ref{lem:middle_structure} we
have
\begin{align*}
\Psi_{n}(x)  &  =H(x)\left(  x^{\frac{n}{p_{1}}}-1\right) \\
&  =-H(x)+H(x)\cdot x^{\frac{n}{p_{1}}}%
\end{align*}
Note
\begin{align*}
\deg\left(  H(x)\right)   &  =\psi(n)-\frac{n}{p_{1}}\\
\mathrm{tdeg}\left(  H(x)\cdot x^{\frac{n}{p_{1}}}\right)   &  =\frac{n}%
{p_{1}}%
\end{align*}
We have
\[
\frac{n}{p_{1}}-\left(  \psi(n)-\frac{n}{p_{1}}\right)  =2\,\frac{n}{p_{1}%
}-\psi(n)=\delta^{-}(n)
\]
If $\delta^{-}(n)\leq0$, then there is nothing to show. If $\delta^{-}(n)>0$
then there is a gap in $\Psi_{n}(x)$ between $x^{\psi(n)-\frac{n}{p_{1}}}$ and
$x^{\frac{n}{p_{1}}}$. Therefore
\[
g(\Psi_{n})\geq\delta^{-}(n)
\]
The theorem has been proved.
\end{proof}

\subsection{Proof of Sufficient condition on $g(\Psi_{n})$
(Theorem~\ref{thm:condition_delta})}

There are two claims in Theorem~\ref{thm:condition_delta}. We will prove them
one by one.

\begin{proof}
[Proof of Theorem~\ref{thm:condition_delta} Claim 1]We will prove that
$g(\Psi_{n})=\delta^{-}(n)$ if $\delta^{-}(n)\geq\frac{1}{2}\frac{n}{p_{1}}$.
From Lemma~\ref{lem:middle_structure} we have
\[
\Psi_{n}(x)=-H(x)+H(x)\cdot x^{\frac{n}{p_{1}}}%
\]
Let%
\[
\delta^{-}\left(  n\right)  =\mathrm{tdeg}\left(  H(x)\cdot x^{\frac{n}{p_{1}%
}}\right)  -\deg\left(  H(x)\right)
\]
Note that if $\delta^{-}(n)\geq\deg\left(  H(x)\right)  $, then we obviously
have $g(\Psi_{n})=\delta^{-}(n)$. In the following we simplify the expression
$\delta^{-}\left(  n\right)  $ and the condition $\delta^{-}(n)\geq\deg\left(
H(x)\right)  $. First, we simplify the expression $\delta^{-}\left(  n\right)
$.%
\begin{align*}
\delta^{-}\left(  n\right)   &  =\mathrm{tdeg}\left(  H(x)\cdot x^{\frac
{n}{p_{1}}}\right)  -\deg\left(  H(x)\right) \\
&  =\frac{n}{p_{1}}-\left(  \psi(n)-\frac{n}{p_{1}}\right) \\
&  =2\frac{n}{p_{1}}-\psi(n)
\end{align*}
Next, we simplify the condition $\delta^{-}(n)\geq\deg\left(  H(x)\right)  $.
\begin{align*}
\delta^{-}(n)\geq\deg\left(  H(x)\right)   &  \iff2\frac{n}{p_{1}}-\psi
(n)\geq\psi(n)-\frac{n}{p_{1}}\\
&  \iff3\,\frac{n}{p_{1}}-2\,\psi(n)\geq0\\
&  \iff\frac{3}{2}\,\frac{n}{p_{1}}-\psi(n)\geq0\\
&  \iff2\,\frac{n}{p_{1}}-\psi(n)\geq\frac{1}{2}\,\frac{n}{p_{1}}\\
&  \iff\delta^{-}(n)\geq\frac{1}{2}\,\frac{n}{p_{1}}%
\end{align*}
Therefore we have shown if $\delta^{-}(n)\geq\frac{1}{2}\,\frac{n}{p_{1}}$
then $g(\Psi_{n})=\delta^{-}(n)$ which proves the first claim of the theorem.
\end{proof}

Before we prove the second claim of Theorem \ref{thm:condition_delta}, we need
a technical lemma.

\begin{lemma}
\label{lem:suff2_maxg_smidg} If $p_{2}>(k-1)(2p_{1}-3)$ then $\delta
^{-}(n)\geq\;\frac{1}{2}\frac{n}{p_{1}}$.
\end{lemma}

\begin{proof}
Note%
\begin{align*}
&  \delta^{-}(n)\;\;\geq\;\;\frac{1}{2}\frac{n}{p_{1}}\\
&  \Longleftrightarrow\;\;\frac{3}{2}\;\frac{n}{p_{1}}\;\;\geq\;\;\psi(n)\\
&  \Longleftrightarrow\;\;\frac{3}{2}\;\frac{n}{p_{1}}\;\;\geq\;\;n-\varphi
(n)\\
&  \Longleftrightarrow\;\;\frac{3}{2}\;\frac{1}{p_{1}}\;\;\geq\;\;1-\frac
{\varphi(n)}{n}\\
&  \Longleftrightarrow\;\;\frac{3}{2}\;\frac{1}{p_{1}}\;\;\geq\;\;1-\left(
1-\frac{1}{p_{1}}\right)  \cdots\left(  1-\frac{1}{p_{k}}\right) \\
&  \Longleftrightarrow\;\;\frac{1}{2}\;\frac{1}{p_{1}}\;\;\geq\;\;1-\frac
{1}{p_{1}}-\left(  1-\frac{1}{p_{1}}\right)  \cdots\left(  1-\frac{1}{p_{k}%
}\right) \\
&  \Longleftrightarrow\;\;\frac{1}{2}\;\frac{1}{p_{1}}\;\;\geq\;\;\left(
1-\frac{1}{p_{1}}\right)  \left(  1-\left(  1-\frac{1}{p_{2}}\right)
\cdots\left(  1-\frac{1}{p_{k}}\right)  \right) \\
&  \Longleftrightarrow\;\;\frac{1}{2}\;\;\geq\;\;\left(  p_{1}-1\right)
\cdot\left(  1-\left(  1-\frac{1}{p_{2}}\right)  \cdots\left(  1-\frac
{1}{p_{k}}\right)  \right) \\
&  \boldsymbol{\Longleftarrow}\;\;\frac{1}{2}\;\;\geq\;\;\left(
p_{1}-1\right)  \cdot\left(  1-\left(  1-\frac{1}{p_{2}}\right)  \left(
1-\frac{1}{p_{2}+1}\right)  \cdots\left(  1-\frac{1}{p_{2}+k-2}\right)
\right) \\
&  \Longleftrightarrow\;\;\frac{1}{2}\;\;\geq\;\;\left(  p_{1}-1\right)
\cdot\left(  1-\left(  \frac{p_{2}-1}{p_{2}}\right)  \left(  \frac{p_{2}%
}{p_{2}+1}\right)  \left(  \frac{p_{2}+1}{p_{2}+2}\right)  \cdots\left(
\frac{p_{2}+k-3}{p_{2}+k-2}\right)  \right) \\
&  \Longleftrightarrow\;\;\frac{1}{2}\;\;\geq\;\;\left(  p_{1}-1\right)
\cdot\left(  1-\frac{p_{2}-1}{p_{2}+k-2}\right) \\
&  \Longleftrightarrow\;\;\frac{1}{2}\;\;\geq\;\;\left(  p_{1}-1\right)
\cdot\frac{k-1}{p_{2}+k-2}\\
&  \Longleftrightarrow\;\;\frac{p_{2}+k-2}{2}\;\;\geq\;\;(k-1)\left(
p_{1}-1\right) \\
&  \Longleftrightarrow\;\;p_{2}+k-2\;\;\geq\;\;(k-1)\left(  2p_{1}-2\right) \\
&  \Longleftrightarrow\;\;p_{2}\;\;\geq\;\;(k-1)\left(  2p_{1}-2\right)
-(k-2)\\
&  \Longleftrightarrow\;\;p_{2}\;\;\geq\;\;(k-1)\left(  2p_{1}-3\right)  +1\\
&  \Longleftrightarrow\;\;p_{2}\;\;>\;\;(k-1)\left(  2p_{1}-3\right)  .
\end{align*}
Therefore, if $p_{2}>(k-1)\left(  2p_{1}-3\right)  $, then $\delta^{-}%
(n)\geq\frac{1}{2}\frac{n}{p_{1}}$.
\end{proof}

\begin{proof}
[Proof of Theorem~\ref{thm:condition_delta} Claim 2]We will prove
\[
\lim_{b\rightarrow\infty}\frac{\#\left\{  n\;:\;p_{k}\leq b,\,p_{1}%
=p,\,\delta^{-}(n)\geq\frac{1}{2}\frac{n}{p_{1}}\right\}  }{\#\left\{
\,n\;:\;p_{k}\leq b,\,p_{1}=p\,\right\}  \hfill}=1
\]
Let $q_{i}$ be the $i$-th odd prime, that is, $q_{1}=3,\ q_{2}=5,\ q_{3}%
=7,\ q_{4}=11$, etc. Let $k\geq2$. Let $p=q_{v}$ and $b=q_{w}$. Then we have
\begin{align*}
&  \#\left\{  n\;:\;p_{k}\leq b,\,p_{1}=p\right\} \\
&  =\#\left\{  (p_{1},\dots,p_{k})\;:\;p_{1}<\dots<p_{k}\leq b,\,p_{1}%
=p\right\} \\
&  =\#\left\{  (q_{i_{1}},\dots,q_{i_{k}})\;:\;q_{i_{1}}<\dots<q_{i_{k}}\leq
q_{w},\ q_{i_{1}}=q_{v}\right\} \\
&  =\#\left\{  (i_{1},i_{2},\dots,i_{k})\;:\;i_{1}<i_{2}<\dots<i_{k}\leq
w,\,i_{1}=v\right\} \\
&  =\#\left\{  (i_{2},\dots,i_{k})\;:\;v+1\leq i_{2}<\dots<i_{k}\leq w\right\}
\\
&  =\#\left\{  (i_{2},\dots,i_{k})\;:\;v+1\leq i_{2}<\dots<i_{k}\leq v+\left(
w-v\right)  \right\} \\
&  =\binom{w-v}{k-1}%
\end{align*}
Thus%
\[
\#\left\{  n\;:\;p_{k}\leq b,\,p_{1}=p\right\}  =\binom{w-v}{k-1}%
\]
Note%
\begin{align*}
&  \#\left\{  n\;:\;p_{k}\leq b,\,p_{1}=p,\,\delta^{-}(n)\geq\;\frac{1}%
{2}\frac{n}{p_{1}}\right\} \\
&  =\#\left\{  (p_{1},\,\dots,p_{k})\;:\;p_{k}\leq b,\,p_{1}=p,\ \delta
^{-}(p_{1}\cdots p_{k})\geq\;\frac{1}{2}\frac{p_{1}\cdots p_{k}}{p_{1}%
}\right\} \\
&  \geq\#\{(p_{1},\dots,p_{k})\;:\;p_{k}\leq b,\,p_{1}=p,\,\ p_{2}%
>(k-1)(2p_{1}-3)\}\ \ \ \ \text{(from Lemma~\ref{lem:suff2_maxg_smidg})}\\
&  =\#\{(q_{i_{1}},\dots,q_{i_{k}})\;:\;q_{i_{1}}<\dots<q_{i_{k}}\leq
q_{w},\,q_{i_{1}}=q_{v},\,q_{i_{2}}>(k-1)(2q_{v}-3)\}\\
&  =\#\{(q_{i_{1}},\dots,q_{i_{k}})\;:\;q_{i_{1}}<\dots<q_{i_{k}}\leq
q_{w},\,q_{i_{1}}=q_{v},\,q_{i_{2}}\geq q_{y}\}\ \ \text{where }%
y=\operatorname*{argmin}\limits_{q_{i}>\left(  k-1\right)  \left(
2q_{v}-3\right)  }i\\
&  =\#\{(i_{1},\dots,i_{k})\;:\;i_{1}<\dots<i_{k}\leq w,\,i_{1}=v,\,i_{2}\geq
y\}\\
&  =\#\{(i_{2},\dots,i_{k})\;:\;v+1\leq i_{2}<\dots<i_{k}\leq w,\,i_{2}\geq
y\}\\
&  =\#\{(i_{2},\dots,i_{k})\;:\;\max\left\{  v+1,\,y\right\}  \leq i_{2}%
<\dots<i_{k}\leq w\}\\
&  =\#\{(i_{2},\dots,i_{k})\;:\;y\leq i_{2}<\dots<i_{k}\leq
w\}\ \ \ \ \text{(since }y\geq v+1\text{)}\\
&  =\binom{w-y+1}{k-1}%
\end{align*}
Thus we have%
\[
\#\left\{  n\;:\;p_{k}\leq b,\,p_{1}=p,\,\delta^{-}(n)\geq\;\frac{1}{2}%
\frac{n}{p_{1}}\right\}  \geq\binom{w-y+1}{k-1}%
\]
Note
\[
\lim_{b\rightarrow\infty}\frac{\#\left\{  n\;:\;p_{k}\leq b,\,p_{1}%
=p,\,\delta^{-}(n)\geq\;\frac{1}{2}\frac{n}{p_{1}}\right\}  \hfill}{\#\left\{
n\;:\;p_{k}\leq b,\,p_{1}=p\right\}  \hfill}\geq\lim_{w\rightarrow\infty}%
\frac{\binom{w-y+1}{k-1}}{\binom{w-v}{k-1}}=\lim_{w\rightarrow\infty}%
\frac{\frac{1}{\left(  k-1\right)  !}w^{k-1}+\cdots}{\frac{1}{\left(
k-1\right)  !}w^{k-1}+\cdots}=1
\]
Since
\[
\lim_{b\rightarrow\infty}\frac{\#\left\{  n\;:\;p_{k}\leq b,\,p_{1}%
=p,\,\delta^{-}(n)\geq\;\frac{1}{2}\frac{n}{p_{1}}\right\}  \hfill}{\#\left\{
n\;:\;p_{k}\leq b,\,p_{1}=p\right\}  \hfill}\leq1
\]
we can conclude%
\[
\lim_{b\rightarrow\infty}\frac{\#\left\{  n\;:\;p_{k}\leq b,\,p_{1}%
=p,\,\delta^{-}(n)\geq\;\frac{1}{2}\frac{n}{p_{1}}\right\}  \hfill}{\#\left\{
n\;:\;p_{k}\leq b,\,p_{1}=p\right\}  \hfill}=1
\]
which proves the second claim of the theorem.
\end{proof}

\section{Evidence for Equivalent condition on $g(\Phi_{n})$~(Conjecture
\ref{conj:condition_phi})}

\label{sec:support}

Let us recall the conjecture: $g(\Phi_{n})=\varphi(p_{1}\cdots p_{k-1})$ if
and only if $p_{k}>p_{1}\cdots p_{k-1}$. 
The conjecture is trivially true for $k=1$. In~\cite{Hong2012}, the conjecture
is proved for $k=2$. For $k\geq3$ the conjecture is still open. One way to
check (support or disprove) the conjecture is to compute $\Phi_{n}$ for many
$n$ with $k\geq3$ and to check whether the maximum gap is $\varphi(p_{1}\cdots
p_{k-1})$ or not. We did this for $n$ up to $40,000$, without finding any
counter-example. However, this method only shows that the conjecture is true
for finitely many such $n$.

In this section, we will describe an algorithm (Algorithm~\ref{alg:modulo})
which allows the conjecture to be checked for infinitely many such $n$ and we
will report that we have done so (Theorem~\ref{thm:verified}).   For the sake of notational
simplicity, let $m=p_{1}\cdots p_{k-1}$  and $p=p_{k}.$ Then the above
conjecture can be restated as: $g(\Phi_{mp})=\varphi(m)$ if and only if $p>m$.   The algorithm
(which will be given later) is based on the following theorem.

\begin{theorem}
[Invariance]\label{thm:modulo} Let $m$ be odd square-free. Let $p,\,p^{\prime
}>m$ be primes such that $p\equiv_{m}p^{\prime}$. Then
\[
g(\Phi_{mp})=g(\Phi_{mp^{\prime}})
\]

\end{theorem}

\begin{proof}
Let $m$ be odd square-free. Let $p>m$ be prime. We will divide the proof into
several steps.

\begin{enumerate}
\item Let $q=\operatorname*{quo}\left(  p,m\right)  $ and
$r=\operatorname*{rem}\left(  p,m\right)  $. Let
\begin{align*}
\Phi_{mp}  &  =\sum_{i=0}^{\varphi(m)-1}\;f_{m,p,i}x^{ip} &  &  \deg
f_{m,p,i}<p\\
f_{m,p,i}  &  =\sum_{j=0}^{q}\;f_{m,p,i,j}x^{jm} &  &  \deg f_{m,p,i,j}<m
\end{align*}
We recall the following results from~\cite{Alkateeb2017}: For all $0\leq
i\leq\varphi(m)-1$, we have

\begin{enumerate}
\item[(C1)] $f_{m,p,i,0}=\cdots=f_{m,p,i,q-1}$

\item[(C2)] $f_{m,p,i,q}=\operatorname*{rem}\left(  f_{m,p,i,0},x^{r}\right)
$

\item[(C3)] $f_{m,p,i,0}=f_{m,p^{\prime},i,0}$ if $p\equiv_{m}p^{\prime}$
\end{enumerate}

\item From $\Phi_{mp}=\sum_{i=0}^{\varphi(m)-1}\;f_{m,p,i}x^{ip}$, we have%
\begin{equation}
g(\Phi_{mp})=\max\left\{  \max_{0\leq i\leq\varphi(m)-1}g(f_{m,p,i}%
),\,\max_{0\leq i\leq\varphi(m)-2}\left(  p+\mathrm{tdeg}(f_{m,p,i+1}%
)-\deg(f_{m,p,i})\right)  \right\}  \label{eq:mod1}%
\end{equation}

\item From $f_{m,p,i}=\sum_{j=0}^{q}f_{m,p,i,j}x^{jp}$ and (C1) and (C2), we
have%
\begin{align}
g(f_{m,p,i})  &  =\max\left\{  g(f_{m,p,i,0}),\,g(f_{m,p,i,q}%
),\,m+\mathrm{tdeg}(f_{m,p,i,0})-\deg(f_{m,p,i,0})\right\} \nonumber\\
&  =\max\left\{  g(f_{m,p,i,0}),\,m+\mathrm{tdeg}(f_{m,p,i,0})-\deg
(f_{m,p,i,0})\right\}  \label{eq:mod2}%
\end{align}

\item From $p-qm=r$, we have%
\begin{align}
p+\mathrm{tdeg}(f_{m,p,i+1})-\deg(f_{m,p,i})  &  =
\begin{cases}
p+\mathrm{tdeg}(f_{m,p,i+1,0})-((q-1)m+\deg(f_{m,p,i,0})) & \text{if }
f_{m,p,i,q}=0 \nonumber\\
p+\mathrm{tdeg}(f_{m,p,i+1,0})-(qm +\deg(f_{m,p,i,q})) & \text{else}%
\end{cases}
\\
&  =
\begin{cases}
r+m+\mathrm{tdeg}(f_{m,p,i+1,0})-\deg(f_{m,p,i,0}) & \text{if }f_{m,p,i,q}=0\\
r+\mathrm{tdeg}(f_{m,p,i+1,0})-\deg(\operatorname*{rem} \left(  f_{m,p,i,0}%
,x^{r}\right)  ) & \text{else}%
\end{cases}
\label{eq:mod3}%
\end{align}

\item Combining the equalities (\ref{eq:mod1}), (\ref{eq:mod2}) and
(\ref{eq:mod3}), we see $g(\Phi_{mp})$ depends \emph{only on} $m,r$ and
$f_{m,p,i,0}$.

\item Let $p^{\prime}>m$ be a prime other than $p$. Then $g(\Phi_{mp^{\prime}%
})$ also depends \emph{only on} $m,r^{\prime}$ and $f_{m,p^{\prime},i,0}$.

\item Suppose $p\equiv_{m}p^{\prime}$. Then obviously $r=r^{\prime}$.
Furthermore from (C3), we have $f_{m,p,i,0}=f_{m,p^{\prime},i,0}$. Thus
$g(  \Phi_{mp})  =g(  \Phi_{mp^{\prime}})  $.
\end{enumerate}
\end{proof}

\noindent From the above theorem (Theorem~\ref{thm:modulo}) we immediately
obtain the following algorithm.

\begin{algorithm}
[Checking the conjecture]\label{alg:modulo} \hspace{10pt}

\begin{description}
\item[In:] $m$, odd square-free, say $m=p_{1}\cdots p_{k-1}$ and $p_{1}%
<\cdots<p_{k}$

\item[Out:] truth of the claim that $\underset{\text{prime }p>p_{k-1}%
}{\forall}\left[  \ g(\Phi_{mp})=\varphi(m)\ \Longleftrightarrow
\ p>m\ \right]  $
\end{description}

\begin{enumerate}
\item for $p$ from $p_{k-1}+1$ to $m-1$, p prime, do

\begin{enumerate}
\item $F\leftarrow\Phi_{mp}$

\item $g\leftarrow$ the maximum gap of $F$

\item if $g=\varphi(m)$ then return false
\end{enumerate}

\item for $r$ from $1$ to $m-1$, where $\gcd(m,r)=1$, do

\begin{enumerate}
\item $p\leftarrow$ the smallest prime larger than $m$ such that
$\mathrm{rem}(m,p)=r$

\item $F\leftarrow\Phi_{mp}$

\item $g\leftarrow$ the maximum gap of $F$

\item if $g\neq\varphi(m)$ then return false
\end{enumerate}

\item return true
\end{enumerate}
\end{algorithm}

\noindent We have implemented the above algorithm in C language. The
cyclotomic polynomials were computed using the algorithm called Sparse Power
Series (Algorithm 4 in~\cite{Arnold2011}) because it is the fastest known
algorithm for inputs where $p$ is not very big compared to $m$. The code for
the algorithm has been kindly provided by Andrew Arnold, one of the authors
of~\cite{Arnold2011}. By executing the program, so far we have proved the following.

\begin{theorem}
[Evidence of the conjecture for infinitely many primes]\label{thm:verified}For
all primes $p$ and $m<1000$, we have%
\[
g(\Phi_{mp})=\varphi(m)\ \ \text{if and only if \ }p>m
\]
In other words, for all $k$ and for all $p_{1},\ldots,p_{k}$ such that
$p_{1}\cdots p_{k-1}<1000,$ we have
\[
g(\Phi_{p_{1}\cdots p_{k}})=\varphi(p_{1}\cdots p_{k-1})\ \ \,\text{if and
only if }p_{k}>p_{1}\cdots p_{k-1}.
\]

\end{theorem}

\noindent The above computation took 86 minutes on a MacBook Pro (CPU: 2.4 GHz
Intel Core i5, Memory: 16 GB 1600 MHz DDR3).  Of course, one could continue
to check larger $m$ values using larger computing resources.

\bibliographystyle{acm}
\bibliography{paper}

\end{document}